\documentclass[12pt,a4paper]{article}
\usepackage[utf8]{inputenc}
\usepackage[T1]{fontenc}
\usepackage{amsmath}
\usepackage{amsfonts}
\usepackage{amssymb}
\usepackage{makeidx}
\usepackage{graphicx}
\usepackage[T1]{fontenc}
\usepackage{csquotes}
\usepackage{hyperref}
\usepackage[french,english]{babel}
\usepackage{amsfonts,amsmath,amssymb}
\usepackage{amsmath,amscd,amsxtra,latexsym,mathrsfs}
\usepackage{amsthm}

\usepackage{cancel}
\usepackage{ulem}

\usepackage{bbold}
\usepackage{cases}
\usepackage{stackengine}
\usepackage{scalerel}
\usepackage{enumitem}
\usepackage{fullpage} 
\usepackage{graphicx}

\numberwithin{equation}{section}

\newtheorem{theorem}{Theorem}[section]    
\newtheorem{proposition}[theorem]{Proposition}    
\newtheorem{lemma}[theorem]{Lemma}   
\newtheorem{remark}{\textbf{{Remark}}}

\newcommand{\pr}{\mathbb{P}}
\newcommand{\esp}{\mathbb{E}}

\newcommand{\eps}{\epsilon}
\newcommand{\A}{\mathcal{A}}

\newcommand{\bone}{\mathbf}

\def\bone{\mathbb 1}

\makeatletter
\def\oversortoftilde#1{\mathop{\vbox{\m@th\ialign{##\crcr\noalign{\kern3\p@}%
				\sortoftildefill\crcr\noalign{\kern3\p@\nointerlineskip}%
				$\hfil\displaystyle{#1}\hfil$\crcr}}}\limits}

\def\sortoftildefill{$\m@th \setbox\z@\hbox{$\braceld$}%
	\braceld\leaders\vrule \@height\ht\z@ \@depth\z@\hfill\braceru$}
\makeatother

\usepackage{authblk}

\newcommand{\R}{\mathbb R}
\newcommand{\N}{\mathbb N}
\newcommand{\Z}{\mathbb Z}
\newcommand{\CC}{\mathbb C}
\newcommand{\bbS}{\mathbb S}
\newcommand{\C}{\mathcal C}
\newcommand{\F}{\mathcal F}

\def\cC{\mathcal C}
\def\cG{\mathcal G}

\newcommand{\E}{\mathbb E}
\newcommand{\ve}{\varepsilon}

\renewcommand{\P}{\mathbb P}

\newcommand{\SA}{{\mathscr A}}

\newcommand{\SK}{{\mathscr K}}

\newcommand{\SN}{{\mathscr N}}

\newcommand{\SW}{{\mathscr W}}

\def\cD{\mathcal{D}}
\def\cC{\mathcal{C}}
\def\PP{\mathbb{P}}
\def\Se{\mathbb{S}}

\newcommand{\Hm}[1]{\leavevmode{\marginpar{\tiny%
$\hbox to 0mm{\hspace*{-0.5mm}$\leftarrow$\hss}%
\vcenter{\vrule depth 0.1mm height 0.1mm width \the\marginparwidth}%
\hbox to 0mm{\hss$\rightarrow$\hspace*{-0.5mm}}$\\\relax\raggedright
#1}}}

\usepackage{xcolor}

\title{Large time behaviour {for} the semigroup of the kinetic Brownian motion in the plane}

\author[1]{Marc Arnaudon}
\author[2]{Magalie Bénéfice}
\author[1]{Michel Bonnefont}
\author[1]{Delphine Féral}
\affil[1]{{Univ. Bordeaux, CNRS, Bordeaux INP, IMB, UMR 5251}, {Talence},
           {F-33400}, 
            {France}}
\affil[2]{{Université de Lorraine, CNRS, IECL}, {F-54000}, {Nancy}, {France}}

\date{\today}

\begin{document}

\maketitle

\begin{abstract}
{ 
We establish an integration by parts formula for the semigroup in time $T>0$ of the kinetic Brownian motion in the Euclidean plane together with its velocity in the circle. The stochastic differential equation of our kinetic Brownian motion is driven here by one real-valued Brownian motion constructed from
an orthonormal basis of $L^2([0,T],\R)$ and an independent sequence of $\SN(0,1)$ random variables. Our method is based on an explicit computation of a Malliavin dual in the Gaussian space. We are mainly interested in large time $T$. From our integration by parts, we obtain   gradient estimates including a reverse Poincaré inequality for the semigroup. As a direct consequence, we also obtain a Liouville property for the generator of the kinetic Brownian motion and its speed: all bounded harmonic functions are constant.}    
\end{abstract}

\tableofcontents

\section{Introduction}
\label{S0}

\subsection{The context of our study}
\label{Sub01}

In this paper, we are interested in regularisation properties for some hypoelliptic diffusion{s}. We shall focus on the simple model of the kinetic Brownian motion on the plane $\R^2$ and more specifically on its long-time behaviour. This kinetic Brownian motion is  the process whose velocity  is given by a Brownian motion  on the circle $\bbS$. We are interested in the full process given by the velocity and the position. It is thus a process on $\bbS \times \R^2$ but it can also be lifted {to} a process on $\R \times \R^2$.

One interest of the kinetic Brownian motion is that, with a suitable $1$-parameter family of renormalisations, it provides an interpolation  between the geodesic flow and the Brownian motion on the manifold;  see e.g. Angst, Bailleul and Tardif \cite{angst-bailleul-tardif:2015}.

Moreover, for our concern,  it is non linear and  quite  degenerate since there is only a 1-dimensional  noise acting on the velocity and thus  it is necessary to  use brackets of step  {three} with the drift to obtain the  
 Hörmander condition. 

 A powerful tool to obtain the desired regularisation properties is  the Bismut-Elworthy-Li  derivative formula for the semigroup.
 Our goal is thus to obtain such a formula and to investigate its long-time behaviour.
It is well known that 
this formula may be obtained through the Malliavin  calculus and the Malliavin covariance matrix. The general idea consists in solving an {infinitesimal control problem to transfer the tangent flow into a Malliavin derivative and then perform} the integration by parts formula using the Malliavin dual.
In the elliptic setting, it  can be analysed {via} 
the Bakry-\'Emery curvature tensor.
But in the degenerate case, since the  notion of curvature is a delicate question, the general  existing formulae established using this approach   are 
less explicit and difficult to analyse, see e.g. Arnaudon and Thalmaier~\cite{Arnaudon-Thalmaier:10} Theorem~3.2 and Theorem~7.1.
In order to obtain  explicit results, such a method has been analysed case by case, see for example Wang and Zhang~\cite{Wang-Zhang:2013} for    some classes of hypoelliptic processes admitting  linear parts into the drift  and  Wang~\cite{Wang-Grushin} for  Grushin operators. Malliavin calculus has also been used by Kusuoka~\cite{Kusuoka} to deal with some larger classes than hypoelliptic diffusions for small times. We also refer to Üstünel for Malliavin calculus for martingale representation of degenerate diffusions~\cite{ustunel:19} with application to some Log-Sobolev inequalities~\cite{ustunel:arxiv}.

A closely related approach is to use a coupling and  Girsanov theorem. The coupling is  again obtained  through the resolution of an associated control problem. {In this sense, the Malliavin approach may be seen as an infinitesimal version of the coupling approach; see \eqref{E24} and \eqref{eq:T+D=0} below for more precision. }
This coupling method has been implemented to deal with the kinetic Fokker-Planck diffusion in Guillin and Wang~\cite{guillin-wang} and with the sub-elliptic Brownian motion on the free step $2$ Carnot groups by the authors in a previous work~\cite{ABBF:25}. The control problem is somehow difficult since one has to couple both the position and the velocity (or the  planar Brownian motions and their Lévy areas) at the same time.

Let us briefly recall some {related methods and} results about {successful and efficient} couplings in hypoelliptic and sub-elliptic situations. {In these cases, regularisation properties are obtained by studying the tail distribution of the first meeting times of two processes starting from different points.}
It started with  
Ben Arous, Cranston and Kendall~\cite{CranstonKolmogorov} and Kendall~\cite{kendall2007coupling,kendall-coupling-gnl} who constructed successful co-adapted couplings. A 
main progress was then made by Kendall and Banerjee~\cite{BanerjeeKolmogorov} and then by Banerjee, Gordina and Mariano in~\cite{banerjee-gordina-mariano:2017} who {used the Karhunen-Loève expansion of the Brownian bridge to construct} efficient
 (finite look-ahead) non-co-adapted successful couplings for Kolmogorov diffusions and Brownian motions on the Heisenberg group; see also Bénéfice \cite{Nonco-adaptedSU(2),CarnotSuccessful}
 and Luo and Neel~\cite{luo2024nonmarkovian} for generalisations.
 
Note that, in~\cite{ABBF:25},  in addition to the integration by parts formula, the authors  perform an efficient non-co-adapted coupling working only  at a given time $T$.  An important point in the construction is the expansion of the driving Brownian motions in a well chosen orthogonal basis of $L^2([0,T])$. Similar ideas also appear in the present work to compute the Malliavin dual.

A difficulty to deal with the kinetic Brownian motion comes from the fact that the position is obtained through a non linear integral of the noise.
Up to our knowledge, long-time behaviour of the Euclidean kinetic Brownian motion has not been investigated, contrarily to the small time, for which Perruchaud~\cite{perruchaud2023} gives  asymptotics of the heat kernel for our precise model.
Small time asymptotics are also present in the related works of Menozzi~\cite{MR4554678} (where step $2$ kinetic processes are considered). Concerning, the long-time behaviour of the kinetic Brownian motion, Baudoin and Tardif \cite{baudoin-tardif:2018} obtained an exponential convergence under the assumption of a Poincaré inequality for the underlying manifold (see also the improvements on the exponential rate  by Kolb,  Weich, Wolf~\cite{KWW} and  Ren, Tao \cite{RTSpectral23}). In \cite{MR4059002}, Baudoin, Gordina and Mariano used a coupling method for obtaining other gradient inequalities for general Kolmogorov diffusions including kinetic Brownian motion. 

After renormalisation in time and space of the kinetic Brownian motion, our model enters in the field of homogenisation.  Much more complex models have been investigated in Li~\cite{LiXM:08, LiXM:18, LiXM:18bis}, Gehringer and Li~\cite{Gehringer-LiXM:21}, Hairer and Li~\cite{Hairer-LiXM:20}, Li and Sieber~\cite{LiXM-Sieber:22}, leading to various convergence results to limit processes, including speed of convergence. In our context, convergence results are obtained in Section~\ref{S3}. Notice the important fact that, {when doing the rescaling}, the vector field governing the so-called fast motion has a vanishing integral in the circle.  As explained in \cite{LiXM:08}, this induces a second order scaling involving a martingale term in the limit. In this situation, the rate of convergence for the models in \cite{LiXM:08} is an open question. Our goal in this paper is different, as we aim to get estimates of the gradient of the semigroup at each fixed time for the non rescaled process. Also, the state space of the Euclidean kinetic Brownian motion is not compact and is not negatively curved. This again places our model in a critical situation. 

Our goal, in this paper, is thus to obtain estimates of the gradient of the semigroup via a Bismut-Elworthy-Li type formula  at a given and fixed non small time $T$.
We will work through the usual Malliavin calculus and the Malliavin covariance approach.   
We will take advantage of a decomposition of the standard Brownian motion 
on an orthonormal basis  to perform a Gaussian integration by parts and compute the dual of the Malliavin derivative.
The main point then is to establish the long-time behaviour of the (reduced) Malliavin matrix and of the dual.
 This will be obtained through the study of homogenisation and the convergence in law of stochastic oscillating integrals. A crucial point being that, if by the standard Hörmander theorem
  the (reduced) Malliavin matrix is invertible and its inverse admits moments of any order, a uniform bound in $T$ of these moments is needed here.
This uniform bound  is a difficult technical point. This will be done through a careful study of oscillating stochastic integrals and proving,  with the help of the recent work~\cite{Bobkov:23} by Bobkov,  that they have sub-Gaussian densities.

\subsection{Organisation and main results}
\label{Sub02}

In Section~\ref{S1gen} we present the model $(X_t^x)_{t\geq 0}$ for the kinetic Brownian motion   on $\bbS \times \CC$ and its lift on $\R\times\CC$ as well as the main strategy for obtaining integration by parts formula, which will lead to estimates of the derivative of the semigroup at time $T>0$.
 This integration by parts formula is given in Equation~\eqref{E23} of Proposition~\ref{P3} using a Malliavin dual $\delta h$:
\begin{equation*}
    (d_xP_Tf, v)=-\E\left[f(X_T^x)\delta h\right].
\end{equation*}

As our construction of the driving Brownian motion is made via an orthonormal basis in the path space and a sequence of Gaussian random variables, a general formula for Malliavin duals is given in Proposition~\ref{P4} with the explicit Formula~\eqref{E46}, extending to infinite dimension well-known integration by parts formulae on Gaussian spaces. Our first main result is  Theorem~\ref{P4bis} where  both the explicit computation of the Malliavin dual  and  an  explicit Bismut-Li-Elworthy type  integration by parts formula   are established.

In particular, the  dual is expressed in terms of (the inverse of) the Malliavin matrix of the kinetic Brownian motion 
and in terms of oscillating integrals of the form 
\begin{equation*}
\int_0^T g(s) e^{i  B_s} ds \,,\, s\in[0,T]
\end{equation*}
where $(B_t)_{0 \leq t\le T}$ is a standard  Brownian motion on $\R$ and where $g$ is a
(possibly random) $\mathcal C^1$ function.

Section~\ref{S3} is devoted to convergences of stochastic oscillating integrals.
A strong law of large numbers is given in Theorem~\ref{T:LFGN}. Then in Theorem~\ref{T2} and Theorem~\ref{T1}, convergences in law for the driving Brownian motion, the renormalised kinetic Brownian motion as well as for several multiple integrals are stated. In particular this immediately leads to a convergence in law of the renormalised Malliavin dual. 

However, for  {studying the integration by parts formulae in large time}, an essential tool is a uniform control in time of all moments of the inverse of the appropriately renormalised reduced Malliavin matrix. This result is stated in Proposition~\ref{P2}. It turns out that, for obtaining this result, oscillating integrals have to be appropriately sliced, each slice having a sub-Gaussian density, and the sum also having sub-Gaussian density thanks to recent results in~\cite{Bobkov:23}. All of this is explained in Lemma~\ref{L3},~\ref{L1},~\ref{LY},~\ref{L2}. 

In Proposition~\ref{P5}, these results are applied to obtain asymptotic estimates of the moments of the Malliavin dual for $T$ large enough:
\begin{equation*}
    \esp[|\delta  h|^q]^{\frac{1}{q}}\le C_p\left(|v_\R|+\frac1{\sqrt T}|v_\CC|\right),\qquad v=\begin{pmatrix}
v_\R\\v_\CC
\end{pmatrix}\in \R\times \CC.
\end{equation*}

This proposition provides an order $T^{-1/2}$ for the vertical gradient of the semigroup, but only constant order for the horizontal gradient.

Section~\ref{S4} contains the main results of the paper, improving in several directions the estimate of the horizontal gradient. Theorem~\ref{T2bis} yields an estimate of order $T^{-1/4}$ for the horizontal gradient of $P_Tf$  and again $T^{-1/2}$ for the vertical gradient, with a control involving $(P_T|f|^p)^{1/p}$, for any $p\in(1,\infty)$ 
 and $T$ large enough:
\begin{equation*}
    \left|\left(d_xP_Tf,v \right)\right|\le C_p\left(P_T|f|^p(x)\right)^{1/p}\left(\frac1{T^{1/4}}|v_\R|+\frac1{\sqrt T}|v_\CC|\right).
\end{equation*}

The method incorporates an adapted Bismut-Elworthy-Li formula to make a transfer from horizontal to vertical gradient. Theorem~\ref{T3} states a {$   T^{-1/2}$ or }
$\ln(T) T^{-1/2}$ rate for the whole gradient of $P_Tf$, with a control in terms of $\|f\|_\infty$
 for $T$ large enough:
\begin{equation*}
    \left|\left(d_xP_Tf,v \right)\right|\le C\|f\|_\infty\frac{\|v\|}{\sqrt T}
\end{equation*}
for the semigroup on $\bbS\times \CC$ and 
\begin{equation*}
    \left|\left(d_xP_Tf,v \right)\right|\le C\|f\|_\infty\left(\frac{\ln(T)}{\sqrt T}|v_\R|+\frac1{\sqrt T}|v_\CC|\right)
\end{equation*}
for the semigroup on $\R\times \CC$.

For this result, reflection coupling for the driving Brownian motion is performed before applying the integration by parts formulae. Finally, Theorem~\ref{T4} provides, as a consequence a Liouville property for the semigroup: all bounded harmonic functions are constant.

\section{Presentation of the model and Malliavin calculus}\label{S1gen}
\subsection{The model on  { $\bbS \times \CC$ and its lift  on $\R\times \CC$} }
\label{S1}

This article is  devoted to the study of the kinetic Brownian motion in the plane. 
The true model is to consider a stochastic process $(\tilde U_t,Z_t)_{t\geq 0}$ which lives in  $\bbS \times \R^2$ where $\tilde U_t$ is  a Brownian motion on the circle and $Z_t$ its integral; that is satisfying:
\[
dZ_t= \tilde U_t dt.
\]
Here, we will often identify $\R^2$ with $\CC$ in the natural way. And it is also  natural to view $\bbS$ as the quotient $^{\textstyle \R}\big/_{ \textstyle 2\pi\Z}$; this identification being made by the bijection: $\theta \in\  ^{\textstyle \R}\big/_{ \textstyle 2\pi\Z} \to e^{i\theta} \in \bbS$.
Moreover, the Brownian motion on the circle $\bbS$ is also given by the complex exponential of a Brownian motion on $\R$. 
With these identifications, the kinetic Brownian motion on the plane is simply the process $(U_t,Z_t)_{t\geq 0}$ in $^{\textstyle \R}\big/_{ \textstyle 2\pi\Z}  \times \CC$ described by the stochastic differential equation: 
\begin{equation}\label{E28}
\left\{ \begin{array}{ccl}
dU_t &=& dB_t\\
dZ_t&=& e^{i U_t} dt,
\end{array}
\right.
\end{equation}
 $(B_t)$ being a real-valued Brownian motion started at $0$.

It is immediate to see that the kinetic Brownian motion can be lifted to a stochastic process
on $\R\times \CC$ sharing the same stochastic differential equation \eqref{E28}.
Since we will often work only infinitesimally, most of our arguments indifferently hold for the two models on $\bbS\times\CC$ and
$\R\times \CC$. For simplicity, unless otherwise specifically said, we choose  to work  on the lifted  model on $\R\times\CC$ and we denote $(U_t,Z_t)_{t\geq 0}$ the solution of \eqref{E28} on $\R\times\CC$.

For later convenience, the solution $\displaystyle\binom{U_t}{Z_t}$ will be represented by a $3$-dimensional column vector and $Z_t\in \CC$ will often be identified with $\displaystyle \binom{{\rm Re}(Z_t)}{{\rm Im}(Z_t)}$.

The solution starting in $x=\begin{pmatrix}u\\z\end{pmatrix}$ is easily given by $X_t^x= \begin{pmatrix}U_t\\Z_t \end{pmatrix}$ with
 \begin{equation}\label{Edef}
\left\{ \begin{array}{ccl}
U_t &=& u  + B_t\\
Z_t&=& z + \int_0 ^t e^{i U_s} ds =z + e^{iu} \int_0 ^t e^{i B_s} ds.
\end{array}
\right.
\end{equation}

Following the model in \eqref{Edef}, in the sequel we will more generally denote by $X_t$ the functional defined on $\mathbb{R}\times\mathcal{C}([0,+\infty[)$ such that  $X_t(x,(B_s)_{s\geq 0})$ is the kinetic Brownian motion $X_t^x$.
We also denote 
\[
dX_t^x= V_0(X_t^x) dt + V dB_t
\]
where the vector fields $V$ and  $V_0$ are given by:
\[
V(u,z)
=
\begin{pmatrix}
1\\0_\CC
\end{pmatrix}
\textrm{ and }
V_0(u,z)
=
\begin{pmatrix}
0\\e^{iu}
\end{pmatrix}.
\]
{
In particular, the generator $L$ written on $\R\times \R^2$ of this Markov process is given by 
\begin{equation}\label{EL}
L= \frac{1}{2} \Delta_u + \cos (u) \partial_{z_1} + \sin (u) \partial_{z_2} 
\end{equation}
Note that with
\[
[V,V_0]= \begin{pmatrix}
0\\ie^{iu}
\end{pmatrix}=:V_1 , \; [V,V_1]= \begin{pmatrix}
0\\-e^{iu}
\end{pmatrix}=:V_2
\]
the vector fields $\{V,V_1,V_2\}$ generate the tangent space in each point and thus the  step-three  Hörmander condition is satisfied and the model is hypoelliptic.
}

\subsection{The tangent flow}
Here we consider the solution starting from different  points with the same Brownian motion {$(B_t)_{t\geq 0}$} (synchronous coupling) {$(X_t^x):=\left(X_t(x,(B_s)_{0\leq s\leq t})\right)$}. The tangent flow is obtained by differentiating this flow with respect to the initial solution:

\begin{equation*}
TX_t^x v:=\left. \frac{d}{da} \right\vert_{a=0} X_t^{x+av} = \begin{pmatrix}
v_0\\
v_z + i v_0 e^{iu} \int_0 ^t e^{i B_s} ds
\end{pmatrix}= J_t^x v
\end{equation*}
where 
$J_t^x:= \begin{pmatrix}
1 &0\\
 i e^{iu} \int_0 ^t e^{i B_s} ds & 1 
\end{pmatrix}$.

For future use note that 

\[(J_t^x)^{-1}= \begin{pmatrix}
1 &0\\
 -  i e^{iu} \int_0 ^t e^{i B_s} ds & 1 
\end{pmatrix}
\textrm{ and }
(J_t^x)^{-1} V= \begin{pmatrix}
1 \\
 -  i e^{iu} \int_0 ^t e^{i B_s} ds 
\end{pmatrix}
\]
and also
\[ J_T^x (J_t^x)^{-1} V= \begin{pmatrix}
1 \\
 i e^{iu} \int_0 ^T e^{i B_s} ds -  i e^{iu} \int_0 ^t e^{i B_s} ds \\ 
\end{pmatrix}
=\begin{pmatrix}
1 \\
 i e^{iu} \int_t ^T e^{i B_s} ds 
\end{pmatrix}.
\]

Notice that both the tangent flow and its inverse are uniformly bounded on each time interval $[0,T]$.

\subsection {Malliavin derivative and integration by parts formula}
In the sequel, we will  only be  interested in computing the Malliavin derivative $\cD X_T^{x}$ for the value at time $T$ of $X_T^{x}
$  which can be seen as the differential of $X_T(x,\cdot)$.
 
 Set $H:=L^2([0,T],\mathbb{R})$.  For $h\in H$:
\begin{align*}
\cD X_T^{x} (h)&: = \left.\frac{d}{d\ve}\right\vert_{\ve=0} X_T\left( {x},
\left( B_t+ \ve \int_ 0^t  h(s) ds \right)_{0\leq t\leq T} \right)=\begin{pmatrix}
\int_0^T h(t) dt\\
i e^{iu} \int_0^T e^{iB_t} \int_0^t h(s) ds dt 
\end{pmatrix}.
\end{align*}

{
 A simple idea to obtain an integration by parts formula is, 
given  a  curve $\gamma(a)$ with $\gamma(0)=x$ and $\dot \gamma(0)=v$, to find $\mathfrak{h}_a$ such that
\begin{equation}
    \label{E24}
X_T \left( \gamma(a), \left(B_t+  \int_ 0^t  \mathfrak{h}_a(s) ds \right)_{0\leq t\leq T}\right)\equiv X_T \left( x, (B_t)_{0\leq t\leq T}\right).
\end{equation}
Taking derivative in $a=0$, this implies
\begin{equation}\label{eq:T+D=0}
TX_T^x v + \cD X_T^{x} (  h)= J_T^x v + \cD X_T^{x} ( h) =0,
\end{equation}
where $\displaystyle  h:=\left.\frac{\partial}{\partial a}\right\vert_{a=0}\mathfrak{h}_a$. However, the curves in $a$ and Equation~\eqref{E24} serve only as a motivation and will {finally} not be needed, only the infinitesimal version~\eqref{eq:T+D=0} will be used.}

From this simple equation, we already obtain in the following proposition our integration by parts formula. However most of the work will be later to find an explicit form for the Malliavin dual,  prove that the integrability assumption is satisfied, and above all,  find an asymptotic as $T\to\infty$.
\begin{proposition}
\label{P3}
Let $f : \R\times \CC \to \R$ be a  {bounded $\mathcal{C}^1$ function with bounded derivative}.  {Denote  by $P_T$ the semigroup of $X_T$ and $\delta$ a Malliavin duality operator. If $ h= h(x,v,T)$  is in the domain of~$\delta$ and satisfies~\eqref{eq:T+D=0}, then we have}  
\begin{equation}
    \label{E23}
    (d_xP_Tf, v)=-\E\left[f(X_T^x)\delta h\right].
\end{equation}
\end{proposition}
As explained in Section \ref{S1}, the same result also holds for the model on $\bbS \times \CC$ with the same equation for the tangent flow, for  the Malliavin derivative and  for the Malliavin dual.

\begin{proof}

On one hand we have 
\begin{equation}
    \label{E26}
     (d_xP_Tf, v){=\left.\frac{d}{d\eps}\right\vert_{\eps=0}\E[f(X_T^{x+\eps v})]}=\E\left[\left(d_{X_T^x}f,J_T^x v\right)\right]
\end{equation}
(and it is {clear} 
that $J_T^x$ has all its moments bounded).
On the other hand, since the derivative of $f$ is bounded, {since} the diffusion $(X_t)$ 
has bounded coefficients as well as their derivatives, 
and since $ h$ is in the domain of~$\delta$ (by hypothesis), we have
\begin{equation}
    \label{E25}
    \E\left[\left(d_{X_T^x}f,\cD(X_T^{x})( h) \right)\right]=\E\left[\cD (f(X_T^{x}))( h)\right]=\E\left[f(X_T^x)\delta h\right].
\end{equation}
From Equation~\eqref{eq:T+D=0} {and the boundedness of the moments of $J_t^x$,} we see that $\left(d_{X_T^x}f,\cD(X_T^{x})(h)\right)$ is integrable, and moreover,   
\begin{equation}
    \label{E27}
    \E\left[\left(d_{X_T^x}f,J_T^x v\right)\right]=-\E\left[\left(d_{X_T^x}f,\cD(X_T^{x})( h)\right)\right].
\end{equation}
From~\eqref{E26},~\eqref{E25} and~\eqref{E27} we get~\eqref{E23}.
\end{proof}

\subsection{The infinitesimal control problem and the Malliavin matrix}

We now describe a solution of Equation \eqref{eq:T+D=0}. By the Malliavin calculus, it is known, {see e.g. \cite{Hairer:11}},  that 
\[
\cD X_T^{x} ( h) = J_T^x \int_0^T (J_t^x)^{-1} V   h(t) dt.
\]

Here, this can be  seen directly since, by integration by parts:
\[
i e^{iu} \int_0^T e^{iB_t} \int_0^t h(s) ds dt  =  i e^{iu} \int_0^T \int_t^T e^{iB_s} ds \;  h(t)  dt. 
\]

As a consequence, Equation \eqref{eq:T+D=0} is equivalent to
\begin{equation}\label{Ih=-v}
\int_0^T  (J_t^x)^{-1} V  h(t) dt =-v.
  \end{equation}
Recalling 
\[(J_t^x)^{-1} V= \begin{pmatrix}
1 \\
 -  i e^{iu} \int_0 ^t e^{i B_s} ds 
\end{pmatrix},
\]
Equation \eqref{Ih=-v} can be written as the system
\begin{equation}\label{eq:syst}
\left\{
\begin{matrix}
\langle h_0, h \rangle= -v_0\\
\langle h_1^{x}, h \rangle =- {v_1}
\\
\langle h_2^{x}, h \rangle =- {v_2}
\\
\end{matrix}
\right.
\end{equation}
where $ \langle \cdot,\cdot\rangle$  denotes the standard scalar product  on {$H$}
and 
\[ h_0(t)= 1, \; h_1^{x}(t)=  \int_0 ^t \sin(u+B_s) ds , \; h_2^{x}(t)= - \int_0 ^t \cos(u+B_s) ds.\]
{Notice that similarly as $J_t^x$, $h_1^x$ and $h_2^x$ depend only on the coordinate~$u$ of~$x$.}

To find a solution, we choose to write 
\begin{equation}\label{E38}
h=\lambda_0 h_0 + \lambda_1 h_1^{x} + \lambda_2 h_2^{x}.\end{equation}

{Note that the}
obtained solution $h$ will {minimise} the norm $\Vert h\Vert_2$ but has no reason to {minimise} the norm of the dual $\Vert \delta h\Vert_2$.

Such a {$ h=h(x,v,T)$} is thus a solution to  the system \eqref{eq:syst} if and only if
\[
\begin{pmatrix}
\lambda_0\\ \lambda_1\\ \lambda_2
\end{pmatrix}
 = - \cC^{x}(T)^{-1} \begin{pmatrix}
v_0\\ v_1\\ v_2
\end{pmatrix}
\]
with $\cC^{x}(T)$ the Gram matrix:
\begin{equation}\label{E36}
\cC^{x}(T)= \begin{pmatrix}
\langle h_0,h_0\rangle& \langle h_0,h_1^{x}\rangle & \langle h_0,h_2^{x}\rangle\\
\langle h_1^{x},h_0\rangle& \langle h_1^{x},h_1^{x}\rangle & \langle h_1^{x},h_2^{x}\rangle\\
\langle h_2^{x},h_0\rangle& \langle h_2^{x},h_1^{x}\rangle & \langle h_2^{x},h_2^{x}\rangle\\
\end{pmatrix}.
\end{equation}

{Note that $\cC^{x}(T)$ may also be written as the reduced Malliavin matrix of $(X_t^x)$ (here $*$ means the transposed matrix):
\begin{equation}\label{E37}
\cC^{x}(T)=\int_0^T  (J_t^x)^{-1} V  V^*   ((J_t^x)^{-1})^*  dt.
\end{equation}
In particular, as $(X_t^x)_t$ satisfies H\"ormander's condition, $\cC^{x}(T)$ is invertible (see \cite{Hairer:11}, Section 4.2). 
Denoting {the random map} $\displaystyle a_T^{x} : H\to \R\times \CC$, $h\mapsto \left(\begin{matrix}
\langle h_0,h\rangle\\\langle h_1^{x}, h\rangle\\\langle h_2^{x},h\rangle
\end{matrix}\right) $, we {observe that} $\cC^{x}(T)$ {is also the matrix of the linear map $a_T^{x} (a_T^{x})^{*}$ with $\displaystyle (a_T^{x})^\ast : \R\times \CC\to H$, $\left(\begin{matrix}
\lambda_0\\ \lambda_1\\ \lambda_2
\end{matrix}\right)\mapsto \lambda_0 h_0+\lambda_1 h_1^x+\lambda_2h_2^x $.}

With this notation at hand, Equation~\eqref{E38} rewrites $ h=(a_T^{x})^{*}(\lambda)$, so the only solution to~\eqref{eq:syst} which is a  linear combination of $(h_0,h_1,h_2)$ is 
\begin{equation}
    \label{E39bis}
    h=-(a_T^{x})^{*}(a_T^{x} (a_T^{x})^{*})^{-1}\begin{pmatrix}
v_0\\ v_1\\ v_2
\end{pmatrix}
\end{equation}

 In fact $a_T^x$ and the matrix $\cC^x$ depend only on $u$ and satisfy:
\begin{equation}\label{ECu}
{a_T^x}=\begin{pmatrix}
 1 & 0\\
 0 & R(u)
\end{pmatrix}a_T^0\quad\text{and}\quad{\cC^x(T)}= \begin{pmatrix}
 1 & 0\\
 0 & R(u)
\end{pmatrix} \cC^0(T)  \begin{pmatrix}
 1 & 0\\
 0 & R(-u)
\end{pmatrix}
\end{equation}
 where $R(u)$ denotes the rotation: $R(u)= \begin{pmatrix} \cos u & -\sin u \\
                 \sin u & \cos u
 \end{pmatrix}.$
 In the remainder of this article, we will simply use the notations $a_T$ and $\cC$ for $a_T^0$ and $\cC^0$ respectively.
 
 In particular, \eqref{E39bis} can be rewritten:
 \begin{equation}\label{E39}
 h=-a_T^{*}\left(a_Ta_T^{*}\right)^{-1}\begin{pmatrix}
 1 & 0\\
 0 & R(-u)
\end{pmatrix}\begin{pmatrix}
v_0\\ v_1\\ v_2
\end{pmatrix}.
 \end{equation}

\subsection{Malliavin calculus in a basis and computation of the dual}
\label{S2.3}

Our method for integration  by parts is based on a {non-co-adapted}  construction of the Brownian motion $(B_t)_{t\in [0,T]}$ in the model~\eqref{E28} where  $T>0$ is fixed. This construction depends on a  choice of an orthonormal basis. Working in a specific basis serves only as an intermediate step to do the computation. At the end, we will see that the computation of the dual will not depend on this choice.

Let $(e_k)_{k\geq 0}$ be an orthonormal basis of $L^2([0,1],\R)$ with {for simplicity} $e_0\equiv 1$  and set  $g_k(t)=\int_0^t e_k(s) ds$ for $k\geq 0$.
We then consider the family  $ (g_{k,T})_{k\geq 0}$ on $H=L^2([0,T],\R)$  obtained from dilation;  that is 
\begin{equation}
    \label{E29}
   \forall k\ge 0,\quad g_{k,T}(t):={\sqrt T}g_k\left(\frac{t}{T}\right),\ t\in[0,T].
    \end{equation}
    In particular, {$(e_{k,T})_{k\geq 0}:=(\dot g_{k,T})_{k\geq 0}$} is an orthonormal basis  of $L^2([0,T],\R)$. We also have
    \begin{equation}\label{E30}
    g_{0,T}(t)= \frac{t}{\sqrt T}, \ t\in[0,T],  \hbox{and}\   g_{k,T}(T)=0 \text{ for all } k\geq 1.
\end{equation}
A standard choice is given by the Karhunen-Loève basis: 
\begin{equation}\label{E41}
g_{0,T}(t)=\frac{t}{\sqrt T}\quad\hbox{and for $k\ge 1$,}\quad 
 g_{k,T}(t)=\frac{\sqrt{2T}}{\pi k}\sin\left(\frac{k\pi t}{T}\right), \;
 { t\in [0,T]}.
\end{equation}

 With a sequence $(\xi_k)_{k\ge0}$ of independent real-valued Gaussian~$\SN(0,1)$ variables, it is well-known that we can construct a Brownian motion $\displaystyle (B_t)_{t\in [0,T]}$ via the formula 
 \begin{equation}
     \label{E31}
     \forall t\in [0,T],\quad B_t=\sum_{k=0}^\infty g_{k,T}(t)\xi_k
 \end{equation}
and  that the convergence of the series is a.s. uniform in $t \in [0,1]$  (see~\cite{ItoNisio}).

 For future use, we also introduce the 
Brownian motion $\displaystyle (\tilde B_t)_{t\in [0,1]}$ via the formula 
 \begin{equation}
     \label{E31bis}
     \forall t\in [0,1],\quad \tilde  B_t=\sum_{k=0}^\infty g_{k}(t)\xi_k,
 \end{equation}
 and  notice that
 \[
 B_{T s} =\sqrt T \tilde B_s, \quad s \in [0,1].
 \]

 Also, any element $h\in H$ has a decomposition 
 \begin{equation}
     \label{E32}
     [0,T]\ni t\mapsto h(t)=\sum_{k\ge 0}\alpha_k\dot g_{k,T}(t),
 \end{equation}
 with $(\alpha_k)_{k\ge 0}$ in the space $\ell^2$ of square integrable sequences and which will always be considered as a column vector. 
 
 With the decomposition~\eqref{E32} of a possibly random $h$, we obtain 
 \begin{equation}
     \label{E33}
     B_t+\int_0^th(s)ds=\sum_{k=0}^\infty g_{k,T}(t)(\xi_k+\alpha_k)
 \end{equation}
 with $(\alpha_k)_{k\ge 0}$ a $\ell^2$-valued random variable.

 \begin{proposition}
 \label{P4}
 Assume that the Brownian motion $(B_t)_{t\in[0,T]}$ is given by~\eqref{E31}: $B_t=\sum\limits_{k=0}^\infty g_{k,T}(t) \xi_k$. 
 Let $h$ be a $L^1(\Omega,H)$ random variable given by 
 \begin{equation}
     \label{E32bis}
      h(\omega)(t)=\sum_{k\ge 0}\alpha_k(\omega)\dot g_{k,T}(t)\quad\hbox{ with for $k\ge 0$},\quad \alpha_k(\omega)=\SA_k\left((\xi_\ell(\omega))_{\ell\ge 0}\right)
 \end{equation}
 for some functions $\SA_k: \R^\N \to \R$.
 Assume that for all {$K\geq 0$,} $ 0\leq k \leq K$,  the random maps
 \begin{equation}
     \label{EP4.6}
     \begin{split}
         \SA_k^{{K}}(\omega) : \R^{K+1}&\to \R\\
         (z_0,\ldots,z_K)&\mapsto
        \SA_k \left(z_0,\ldots,z_K,\xi_{[K+1,\infty)}(\omega)\right)
     \end{split}
 \end{equation}
 are {integrable and with integrable derivative on $\mathbb{R}^{K+1}$ with respect to the standard Gaussian distribution, a.s. in $\xi_{[K+1,\infty)}$}, where we denoted 
  $\xi_{[K+1,\infty)}(\omega):=(\xi_{K+1}(
  \omega),\xi_{K+2}(\omega),\ldots)$. Let 
  \begin{equation}\label{E46}
  \mathfrak{d} h:=\sum_{k\ge 0}\left(\alpha_k\xi_k-\frac{\partial \alpha_k}{\partial\xi_k}\right) 
  \end{equation}
 and assume that this series converges in $L^1(\Omega,\R)$.
 
Then $h$ is in the domain of the adjoint of the Malliavin derivative $\delta$ and  $\delta h$ is given by $ \mathfrak{d} h$, that is, for any { $\cC^1$ function $f : \R\times \CC\to \R$, bounded with bounded derivative }
we have the integration by parts 
\begin{equation}
    \label{EP4.3}
    \E\left[\left(d_{X_T^x}f,\cD(X_T^{x})( h)\right)\right]=\E\left[f(X_T^x){\mathfrak{d}}
    h\right].
\end{equation}
 \end{proposition}

 \begin{proof} The proof will proceed by finite dimensional approximations. For $K\geq 0$, 
 let \[
 h^K:=\sum_{k=0}^K\alpha_k\dot g_{k,T}
\] 
and set $$\begin{array}{ccccc}
F^K & : & \R^{K+1}&\to &\R\\
 && (z^0,\ldots,z^K)&\mapsto& F^K(z^0,\ldots,z^K) \\
\end{array}$$
  the  random $\xi_{[K+1,\infty)}$-measurable function such that $F^K(\xi_0,\ldots,\xi_K)=f(X_T^{x})$. 
  
  With these notations, the Malliavin derivative in the direction $h^K$ writes:
\begin{align*}
\left(d_{X_T^{x}}f,\cD(X_T^{x})(h^K) \right)&{=\left.\frac{d}{d\eps}\right\vert_{\eps=0}f\left(X_T\left(x,\left(B_t+\eps\sum\limits_{k=0}^K\alpha_kg_{k,T}(t)\right)\right)\right)}\\
&{=\left.\frac{d}{d\eps}\right\vert_{\eps=0}F^K\left((\xi_k+\eps\alpha_k)_{0\leq k\leq K}\right)}\\&=\langle\nabla F^K, \SA^K\rangle(\xi_0,\ldots,\xi_K)
\end{align*}
  where $\SA^K$ is the random map
whose coordinates are defined in~\eqref{EP4.6}.
 Now letting $\phi_K(z^0,\ldots,z^K)$ be the centered Gaussian density with covariance $I_{K+1}$ in $\R^{K+1}$ we have, by integration by parts and using $\displaystyle \nabla \phi_K(z^0,\ldots,z^K)=-\phi_K(z^0,\ldots,z^K)\cdot(z^0,\ldots,z^K)$, 
 \begin{align*}
     &\E\left[\left(d_{X_T^{x}}f,\cD(X_T^{x})(h^K)\right)|\xi_{[K+1,\infty)}\right]\\&=\int_{\R^{K+1}}\langle\nabla F^K, \SA^K\rangle(z^0,\ldots,z^K)\phi_K(z^0,\ldots,z^K)dz^0\ldots dz^K\\
     &=\int_{\R^{K+1}}F^K(z^0,\ldots,z^K)\left(\langle (z^0,\ldots,z^K),\SA^K(z^0,\ldots,z^K)\rangle
     -{\rm div}(\SA^K)(z^0,\ldots,z^K)
     \right)\\
     &\qquad\qquad\times \phi_K(z^0,\ldots,z^K)dz^0\ldots dz^K\\
     &=\E\left[f(X_T^{x})\mathfrak{d} h^K|\xi_{[K+1,\infty)}\right]
 \end{align*}

 with 
 \[
\mathfrak{d} h^K :=\sum_{k= 0}^K\left(\alpha_k\xi_k-\frac{\partial \alpha_k}{\partial\xi_k}\right).
 \]

 Taking expectation we obtain 
 \begin{equation}
     \label{EP4.5}
     \E\left[\left(d_{X_T^{x}}f,\cD(X_T^{x})(h^K)\right)\right]=\E\left[f(X_T^{x}) \mathfrak{d}  h^K\right].
 \end{equation}
{By hypothesis, $h\in L^1(\Omega,H)$ and $\mathfrak{d}h\in L^1(\Omega,\mathbb{R})$, thus,} 
letting $K\to \infty$, 
$h^K$ converges to 
 $h$ in $L^1(\Omega,H)$ and 
 $ \mathfrak{d}  h^K$ converges to $ \mathfrak{d}  h$ in $L^1(\Omega,\R)$. The expected result then follows from \eqref{EP4.5}.
 \end{proof}

 In order to be able to compute the Malliavin dual of Equation~\eqref{E23} with the help of Proposition~\ref{P4}, we need to express the solution $h$ found in Equation~\eqref{E39} in the chosen basis. For this, it is sufficient to find the matrix of $a_T$. 

We compute the column vectors, {recalling} $e_{k,T}:=\dot g_{k,T}$, $k\ge 0$:{
\begin{equation}
    \label{E40}
    I_0(T):=a_T(e_{0,T})=\begin{pmatrix}\sqrt{T}\\i\int_0^Te^{iB_t} \left(g_{0,T}(t)-{\sqrt T}\right) dt \end{pmatrix} =\begin{pmatrix}\sqrt{T}\\{\sqrt T}{i
    }\int_0^Te^{iB_t} \left(g_0(t/T)-1\right) dt \end{pmatrix}\end{equation}}
    and for $k\ge 1$,
    \begin{equation}\label{E42}
    I_k(T):=a_T(e_{k,T})=\begin{pmatrix} 0\\  i
    \int_0^T e^{iB_t} g_{k,T}(t)dt\end{pmatrix}
     =\begin{pmatrix} 0\\ \sqrt T i
    \int_0^T e^{iB_t} g_{k}(t/T)dt\end{pmatrix}.
\end{equation}
The last formulae are obtained by an integration by parts.

 Now define 
\begin{equation}
    \label{E43}
    \A_T=\left(I_0(T)\ I_1(T)\ \ldots \ I_k(T)\ \ldots\right).
\end{equation}
It is the matrix of {the linear map $a_T :  H\to\R\times \CC$} in the orthonormal basis $(e_{k,T})_{k\ge 0}$ {and therefore satisfies} 
\begin{equation}
    \label{E44}
    \A_T\A_T^{*}= a_Ta_T^{*}=\cC(T).
\end{equation}
Hence $\A_T\A_T^{*}$ is invertible using Malliavin calculus.

It is now possible to write the solution  of Equation~\eqref{E39} in coordinates in the orthonormal basis; that is, $h=\sum\limits_{k\geq 0} \alpha _k e_{k,T}$ is the solution  of Equation~\eqref{eq:T+D=0} given by 
\eqref{E39} if and only if 
\begin{equation}
    \label{E45}
    \alpha=-\A_T^{*} (\A_T\A_T^{*})^{-1}\begin{pmatrix}
 1 & 0\\
 0 & R(-u)
\end{pmatrix}\begin{pmatrix}
v_0\\ v_1\\ v_2
\end{pmatrix}.
\end{equation}

\begin{theorem}\label{P4bis}
Let $T>0$, $x=\begin{pmatrix}
 u\\ z
\end{pmatrix} \in \R \times \CC$, $v\in \R\times \CC$. Consider $h=h(x,v,T)$ defined by \eqref{E39}, then $h$ is {in  $L^p(\Omega,H)$ for any $p\geq 1$} and is also in the domain of $\delta$. Its Malliavin dual $\delta h=\delta h(x,v,T)$ is given by
\begin{equation}\label{EDual}
\begin{split}
   {-} \delta h &=\begin{pmatrix} B_T\\  -i
    \int_0^T e^{iB_t}(B_T-B_t) dt\end{pmatrix}^{*}\cC(T)^{-1}\begin{pmatrix}
 1 & 0\\
 0 & R(-u)
\end{pmatrix}v\\
    &+\int_0^T\int_0^t \begin{pmatrix} s\\  i
    \int_0^s e^{iB_{\rho}}(\rho-s)d\rho\end{pmatrix}^{*}\cC(T)^{-1}\SK(t,s)\cC(T)^{-1}\begin{pmatrix}
 1 & 0\\
 0 & R(-u)
\end{pmatrix}v\,dsdt
\end{split}
\end{equation}
where the matrix $\SK$ is defined by:
\begin{equation}\label{EK}
    \SK(t,s)=\begin{pmatrix}
     0 & {e^{i
     B_s}}^{*}\\
     e^{i
     B_s} & -e^{i
     B_s}\left(i \int_0^t
     e^{i
     B_r}dr\right)^{*}-i \int_0^t
     e^{i
     B_r}dr{e^{i
     B_s}}^{*}
    \end{pmatrix}.
    \end{equation}
{The Malliavin dual $\delta h$ given by \eqref{EDual} is also in $L^p(\Omega,\R)$ for any $p\geq 1$.}

{
Moreover for  any bounded measurable function $f : \R\times \CC\to \R$ and any $T>0 , x\in \R\times \CC$ and $v\in \R\times \CC$, one has the following Bismut-Elworthy-Li type formula: \begin{equation}
    \label{IPP}
    (d_xP_Tf, v)=-\E\left[f(X_T^x)\delta h\right]
\end{equation}
where $\delta h= \delta h(x,v,T)$ is given by \eqref{EDual}.
}
\end{theorem}

\begin{proof} 
{We proceed in five steps. In a first step, we prove $h$ is in $L^p(\Omega,H)$ for any $p\in[1,\infty)$. In a second step
we show that the series  $\mathfrak{d}h$ is well defined and   is equal to \eqref{EDual}. In a third step we establish the $L^p$ convergence of the series defining $\mathfrak{d} h$. 
In a fourth step, we prove that the Bismut-Elworthy-Li type formula \eqref{IPP}
is valid for any  bounded $\cC^1$ function with bounded derivative. For this we need to  introduce some regularised random variables $h^\ve$, do the Malliavin integration by parts and pass to the limit.
In the final step, we extend the the Bismut-Elworthy-Li type formula  \eqref{IPP} to any measurable  bounded function.}

{\underline{Step 1:}}
{We first show that $h$ is in $L^p(\Omega,H)$ for any $p\in[1,\infty)$. 
We have
\[
\Vert h\Vert_H^2= \sum_{k\geq 0} \alpha_k^2 \leq 
\sum_{k\geq 0} \|I_k(T)\|^2 \; \Vert \cC(T)^{-1} \Vert^2 \;   \|v\|^2.
\]
where  the $I_k(T)$ for $k\geq 0$ are given by \eqref{E40} and \eqref{E42}.
By the Cauchy-Schwarz inequality, one has 
$$\|I_k(T)\|^2\leq T\int_0^T g_{k,T}(t)^2dt+T(1+T^{2})\delta_{k=0} $$
and since  $\sum\limits_{k\geq 0}g_{k,T}(t)^2=t$, one has: $$C_T:= \sum\limits_{k\geq 0}\|I_k(T)\|^2\leq T\int_0^T tdt+T(1+T^{2})<+\infty.$$
In particular:
\[
\Vert h\Vert_H\leq C_T^{1/2}
 \; \Vert \cC(T)^{-1} \Vert \;   \|v\|.
\]
and $h\in L^p(\Omega,H)$ for any $p\in[1,\infty)$
since by Malliavin calculus $\cC(T)$ is a.s. invertible with  inverse moments of any orders. 
}

{\underline{Step 2:}} We consider here  only the case $u=0$. In view of  \eqref{E45}, the general case is obtained by rotation. We start with the computation of
$
 {-} \mathfrak{d} h={-}\sum\limits_{k\ge 0}\left(\alpha_k\xi_k-\frac{\partial \alpha_k}{\partial\xi_k}\right).
 $

We look at the first term in $-\mathfrak{d}h$:
\begin{equation}
\label{first-term}
\begin{split}
-\sum\limits_{k\ge 0}\alpha_k\xi_k&=\sum\limits_{k\geq 0}\left(I_k(T)\xi_k\right)^{*}\cC(T)^{-1}v\\
&=\begin{pmatrix} \sqrt{T}\xi_0\\  i
\left(\sum\limits_{k\geq 0}\int_0^T e^{iB_t} \xi_kg_{k,T}(t)dt-\sqrt{T}\int_0^T e^{iB_t}\xi_0dt\right)\end{pmatrix}^{*}\cC(T)^{-1}v.
\end{split}
\end{equation}
By construction of the Brownian motion, we have
the  a.s. uniform convergence on $[0,T]$  of $\sum\limits_{k\geq 0}\xi_kg_{k,T}(t)$ to $B_t$ and thus  the a.s.
convergence of 
$\sum\limits_{k\geq 0}\int_0^T e^{iB_t} \xi_kg_{k,T}(t)dt$ to $\int_0^T e^{iB_t}B_tdt$.
Thus:
\begin{equation}\label{Ederivee1}
{-}\sum\limits_{k\ge 0}\alpha_k\xi_k=\begin{pmatrix} \sqrt{T}\xi_0\\  i\int_0^T e^{iB_t}(B_t-\sqrt{T}\xi_0) dt\end{pmatrix}^*\cC(T)^{-1}v=\begin{pmatrix} B_T\\  -i
\int_0^T e^{iB_t}(B_T-B_t) dt\end{pmatrix}^* \cC(T)^{-1}v.
\end{equation}

We now turn to the second term of $\mathfrak{d}h$. We have
\begin{equation}\label{Ederivee2}{-}\sum\limits_{k\geq 0}\frac{\partial \alpha_k}{\partial\xi_k}=\sum\limits_{k\geq 0}\frac{\partial I_k(T) }{\partial\xi_k}^{*}\cC(T)^{-1}v-\sum\limits_{k\geq 0}I_k(T)^{*}\cC(T)^{-1}\frac{\partial\cC(T) }{\partial\xi_k}\cC(T)^{-1}v.\end{equation}
As $\frac{\partial I_k(T)}{\partial \xi_k}=-\begin{pmatrix} 0\\  
\int_0^T e^{iB_t}g_{k,T}(t)^2 dt\end{pmatrix}+\begin{pmatrix} 0\\
\int_0^Te^{iB_t}\sqrt{T}g_{0,T}(t)dt\end{pmatrix}\delta_{\{k=0\}}$,
 \begin{align*}\sum\limits_{k\geq 0}\frac{\partial I_k(T)}{\partial\xi_k}
&=\begin{pmatrix} 0\\  -\int_0^T e^{iB_t}\left(\sum\limits_{k\geq 0}g_{k,T}(t)^2-t\right) dt\end{pmatrix}.\end{align*}
 Notice that 
\begin{equation}\label{ESum}
\sum\limits_{k\geq 0}g_{k,T}(t)^2=\sum\limits_{k\geq 0}\langle \dot g_{k,T},\mathbb{1}_{[0,t]}\rangle^2=\|\mathbb{1}_{[0,t]}\|_2^2=t,\end{equation}
where $\langle~,~\rangle$ and $\|\cdot\|_2$ still denote the scalar product and the norm on $H$.
 
Then the first term in \eqref{Ederivee2} finally vanishes. 

Using \eqref{E37} and remarking that $\frac{\partial }{\partial \xi_k}(J_t^0)^{-1}V=\begin{pmatrix} 0\\  
\int_0^t e^{iB_s}g_{k,T}(s) ds\end{pmatrix}$, we obtain
\begin{equation}
    \frac{\partial \cC(T)}{\partial \xi_k}=\int_0^T\int_0^tg_{k,T}(s)\SK(t,s)dsdt
\end{equation}
where  $\SK(t,s)$ is given by \eqref{EK}. 

We also have 
\begin{equation*} I_k(T)g_{k,T}(s)=\begin{pmatrix} 0\\  i
\int_0^T e^{iB_{\rho}}g_{k,T}(\rho)g_{k,T}(s)d\rho\end{pmatrix}+\begin{pmatrix} s\\  -i
\int_0^Tse^{iB_{\rho}}d\rho\end{pmatrix}\delta_{\{k=0\}}.
\end{equation*}
Noticing that
 $\sum\limits_{k\geq 0} g_{k,T}(\rho) g_{k,T}(s)=\langle \mathbb{1}_{[0,\rho]},\mathbb{1}_{[0,s]}\rangle=\rho\wedge s$ and using Fubini,
we get $$\sum\limits_{k\geq 0} I_k(T)g_{k,T}(s)=\begin{pmatrix} s\\  i
\int_0^T e^{iB_{\rho}}(\rho\wedge s-s)d\rho\end{pmatrix}=\begin{pmatrix} s\\  i
\int_0^s e^{iB_{\rho}}(\rho-s)d\rho\end{pmatrix}.$$
Finally we obtain the convergence and the  desired expression for $\mathfrak{d}h$.

\underline{Step 3:} 
Let us establish the $L^p$ convergence of the series $
  \mathfrak{d} h=\sum\limits_{k\ge 0}\left(\alpha_k\xi_k-\frac{\partial \alpha_k}{\partial\xi_k}\right)
 $. For the term $\sum\limits_{k\ge 0}\alpha_k\xi_k$, looking at the second line of~\eqref{first-term} and using the fact that $\cC(T)^{-1}$ has finite moments of all orders, we only need to prove that 
\[
\sum\limits_{k\geq K}\int_0^T e^{iB_t} \xi_kg_{k,T}(t)dt\xrightarrow[K\to \infty]{ L^p } 0.
\]
Letting $C_p=\E[|N|^p]$ with~$N$ a $\SN(0,1)$ random variable we obtain
\begin{align*}
    &\E\left[\left|
    \sum\limits_{k\geq K}\int_0^T e^{iB_t} \xi_kg_{k,T}(t)dt
    \right|^p\right]\le T^{p-1}\int_0^T
    \E\left[\left|
    \sum\limits_{k\geq K} \xi_kg_{k,T}(t)
    \right|^p\right]
    dt\\
    &\le C_pT^{p-1}\int_0^T\left(\sum\limits_{k\geq K} g_{k,T}^2(t)\right)^{p/2}dt
    \xrightarrow[K\to \infty]{} 0
\end{align*}
by the dominated convergence theorem and since $\sum\limits_{k\geq 0} g_{k,T}^2(t )=t$. 
Using 
\[\sum\limits_{k\geq 0} |g_{k,T}(\rho) g_{k,T}(s)| \leq \left(\sum\limits_{k\geq 0} g_{k,T}^2(\rho ) \right)^{1/2}\left(\sum\limits_{k\geq 0} g_{k,T}^2(s ) \right)^{1/2} =\sqrt{\rho s},\]
the $L^p$ convergence of the series $\sum\limits_{k\ge 0}\frac{\partial \alpha_k}{\partial\xi_k}$
is similar.

\underline{Step 4:} 
Let us prove the integration by parts  \eqref{EP4.3} for a bounded function $\mathcal{C}^1$ with bounded derivative. The idea is to use Proposition \ref{P4}.
But since the matrix $\mathcal{C}(T)$ is only almost surely invertible, the regularity of the random maps  $\SA_k^K(\omega)$ of  equation \eqref{EP4.6}  for  $h$ and $\alpha$ given by given by \eqref{E45}  is not completely clear.
We thus introduce the regularised random variables $h^{\eps}:=\sum\limits_{k\geq 0}\alpha_k^{\eps}\dot{g}_{k,T}$ where $\eps>0$ and 
\begin{equation}
    \alpha^{\epsilon}:=-\A_T^{*} (\cC(T)+\eps Id)^{-1}\begin{pmatrix}
 1 & 0\\
 0 & R(-u)
\end{pmatrix}\begin{pmatrix}
v_0\\ v_1\\ v_2
\end{pmatrix}.
\end{equation}
 Since $\cC(T)+\eps I_d$ is  invertible for any entries $(z_0,\hdots,z_K)\in\mathbb{R}^{K+1}$
and since  all the entries of $(I_k(T))_{k\geq 0}$ and $\cC(T)$ are almost surely analytic in $(\xi_k)_{k\geq 0}$, 
the associated functions $\SA_k^{K,\eps}(\omega):\mathbb{R}^{K+1}\to\mathbb{R}$
are well defined and almost surely
 analytic in $(z_0,\hdots,z_K)\in\mathbb{R}^{K+1}$.

Now since $\|(\cC(T)+\eps I_d)^{-1}\| \leq \frac 1 \eps$, with similar computation as in Step 1, one gets that 
$h^\ve$ is bounded in $H$ and thus belongs to any $L^p(\Omega,H)$ for any $p\in [1,\infty]$.

Similarly, it is easy to obtain Steps 2 and 3 for $h^\ve$;
that is the well definiteness, the computation and  the $L^p$ convergence of 
\[
\mathfrak{d} h^\eps=\sum\limits_{k\ge 0}\left(\alpha_k^\eps\xi_k-\frac{\partial \alpha_k^\eps}{\partial\xi_k}\right).
\]
Note that $\mathfrak{d} h^\eps$ is obtained by a formula similar to the right hand side of \eqref{EDual} where each occurrence of $\cC(T)^{-1}$ is  replaced by $(\cC(T)+\eps I_d)^{-1}$.

Using Proposition \ref{P4}, we thus obtain  the integration by parts  formula \eqref{EP4.3} for $h^{\eps}$:
\begin{equation}
    \label{EP4.3bis}
    \E\left[\left(d_{X_T^x}f,\cD(X_T^{x})( h^\eps)\right)\right]=\E\left[f(X_T^x){\mathfrak{d}}
    h^\eps\right].
\end{equation}

Now since $\Vert (\cC(T)+\ve)^{-1}\Vert \leq \Vert \cC(T)^{-1}\Vert$  and since by Malliavin calculus the moments of $\cC(T)^{-1}$
are finite for any $p\geq 1$, by the dominated convergence, 
 $h^{\eps} \xrightarrow[\eps \to 0]{} h$ in $L^p(\Omega,H)$ and  
 $\mathfrak{d}h^{\eps} \xrightarrow[\eps \to 0]{} \mathfrak{d}h$ in $L^p(\Omega,\mathbb{R})$, for all $p\geq 1$. Letting $\ve \to 0$ gives that $h$ is in the domain of $\delta$; that is the integration by parts formula  \eqref{EP4.3} is valid for $h$.

Now since, 
\[
\left(d_x P_T f,v\right)=  \E\left[\left(d_{X_T^x}f,\cD(X_T^{x})( h)\right)\right],
\]
one deduces the Bismut-Elworthy-Li formula \eqref{IPP} when $f$  is a bounded $\cC ^1$ function with bounded derivative.

 {\underline{Step 5:}} Let us next prove that \eqref{IPP} is still true for any $f$ measurable and bounded.

For this, 
let $f_n\in \cC_c^\infty$ such that 
$f_n$ simply converges to $f$ and with $\Vert f_n \Vert_\infty \leq \Vert f \Vert_\infty $.

By the dominated convergence theorem, it is clear that for all $x\in \R\times \CC$:
\begin{equation}\label{Ef}
   P_t f_n(x) \xrightarrow[n\to \infty]{}P_t f(x). 
\end{equation}
 
From now, our goal is  to show that the differential of $P_t f_n$ converges uniformly on compact sets. More precisely, we  show just below that for any compact set $K$ of $\R\times \CC$:
\begin{equation}\label{Edf}
\sup_{x\in K} \sup_{\|v\|\leq 1}
\big| \left(d_x P_T f_n,v\right) - \E\left[ f(X_T^x) \delta h(x,v,T)\right] \big| \xrightarrow[n\to \infty]{} 0.
\end{equation}
     
Since $f_n$ is a bounded $\cC^1$ function with bounded derivative, the integration by parts formula \eqref{EP4.3} is valid for $f_n$, and thus using also Cauchy-Schwarz inequality, one has:
 \begin{align*}
     \big| \left(d_x P_T f_n,v\right) - \E\left[ f(X_T^x) \delta h(x,v,T)\right] \big|
     & = \big| \esp[f_n(X_t^x)\delta h(x,v,T)]-\esp[f(X_T^x)\delta h(x,v,T)]\big|\\
     &\leq \esp[|f_n(X_T^x)-f(X_T^x)|^2]^{\frac{1}{2}} \;  \esp[\delta h(x,v,T)^2]^{\frac{1}{2}}. \end{align*}

First, it is clear that by formula \eqref{EDual} and since the Malliavin matrix  $\cC^x(T)$ admit  negative moments of order 2 which are here {uniform in $x\in \R\times \CC$}.
\[
\sup_{x\in \R \times \CC} \sup_{\|v\| \leq 1} \esp[\delta h(x,v,T)^2]<+\infty.
\]

Now, let $K$ be a compact set of $\R\times \mathbb C$.
There exists $R_0>0$ such that 
$K\subset B_{R_0}:=\{ x=(u,z),|u|\leq R_0, |z|\leq R_0\}$.
Let $x=(u,z)\in K$ and $R\geq R_0+T$, one has
\begin{align*} 
\esp[|f_n(X_T^x)-f(X_T^x)|^2]&=\esp\left[|f_n(X_T^x)-f(X_T^x)|^2 \bone_{\{X_T^x\in B_R\}}\right]+ \esp\left[|f_n(X_T^x)-f(X_T^x)|^2 \bone_{\{X_T^x\notin B_R\}}\right].
\end{align*}
For the first term, we have: 
\begin{align*}
 \sup_{x\in K}\E \left[|f_n(X_T^x)-f(X_T^x)|^2 \bone_{\{X_T^x\in B_R\}}\right] &\leq \sup_{x\in K}\int_{B_R} \left| f(y)-f_n(y)\right|^2 p_T{(x,y)} \, dy\\
& \leq \int_{B_R} \left| f(y)-f_n(y)\right|^2 \sup_{x\in K,y\in B_R} p_T{(x,y)}  \, dy\\
 &\hspace{1cm} \xrightarrow[n\to \infty]{}0.
\end{align*}
Here we used  the dominated convergence theorem and the fact  that by Malliavin calculus, the semigroup at time $T$ admits a smooth (here continuous is enough) kernel $p_T$.
For the second term, we have:
\begin{align*}
 \sup_{x\in K} \esp\left[|f_n(X_T^x)-f(X_T^x)|^2 \bone_{\{X_T^x\notin B_R\}}\right]
& \leq 4 \|f\|_{\infty}^2  \sup_{x\in K}\pr\left(X_T^x\notin B_R\right)\\
&\leq 4 \|f\|_{\infty}^2\pr\left(|B_T|\geq R-R_0\right).
\end{align*}
Here we have used that $U_T^x = u+ B_T$ and that  $|Z_T^x|\leq |z| + T \leq R$  for $x\in K$.
As a consequence,
\[
\limsup_{n\to\infty}  \sup_{x\in K}
\esp[|f_n(X_T^x)-f(X_T^x)|^2]\leq 4 \|f\|_{\infty}^2\pr\left(|B_T|\geq R-R_0\right) \xrightarrow[R\to \infty]{}0.
\]
This ends the proof of \eqref{Edf}.

Now it is well known and classical that the conjunction of 
\eqref{Ef} and \eqref{Edf} implies that $P_t f$ is differentiable and satisfies:
\[
\left(d_x P_T f,v\right) = \E\left[ f(X_T^x) \delta h(x,v,T)\right]
\]
which proves that~\eqref{IPP} is valid for any bounded measurable functions~$f$.
\end{proof}

\section{Stochastic oscillating integrals}
\label{S3}

Our goal in this section is now to investigate the asymptotic  behaviour of the terms appearing in the matrix $\cC^x(T)$  and more generally into the dual  $\delta h(x,v,T)$.
As before, it is sufficient to consider the case $x=0$.
We shall  thus focus on  the stochastic integrals  appearing in the quantities $I_k(T)$ in \eqref{E40} and \eqref{E42}. 
 By a change of variable and the definition of $\tilde B_s$  in \eqref{E31bis}, one has
  \begin{equation}\label{E42bis}
  \sqrt T \int_0^T e^{iB_t} g\left(\frac{t}{T}\right)dt=  T^{3/2} \int_0^1 e^{i\sqrt T \tilde B_s} g(s)ds.
  \end{equation}
  In order to treat all the terms in the dual, we also have to consider  iterated integrals of the form:
  { \begin{equation}\label{E42ter}
  \left(\int_0^t  \left(\int_{r=0}^s \ell_2(r) e^{i\sqrt T \tilde B_r} dr\right)     \ell_1(s) \left(e^{i\sqrt T \tilde B_s}\right)^* ds\right)_{0\leq t\leq 1}.
  \end{equation}}

First we give an ergodic theorem (or a strong law) for these stochastic  oscillating integrals. It might be seen as a stochastic version of the Riemann Lebesgue theorem. 
\begin{theorem}\label{T:LFGN}
Let $(\tilde B_t)_{0 \leq t\le 1}$ be a standard  Brownian motion on $\R$. 
{Let $g \in  L^1([0,1], \R^K)$}. Then
\begin{equation}\label{E:LF}
\int_0^t g(s) e^{i\sqrt T \tilde B_s} ds \xrightarrow[T\to \infty]{ {a.s.} } 0 \quad\hbox{uniformly in $t\in[0,1]$}.
\end{equation}
\end{theorem}

 We now turn to the convergence in law of  the above integrals and iterated integrals. This convergence in law  is given in Theorem \ref{T1} below. It   may be seen as a central limit theorem for additive functional of the Brownian motion.  
 
 The idea to obtain Theorem \ref{T1} below  is to transform these additive functionals into martingales. Using  Itô formula, one can  replace,  up to negligible terms, the integration against 
 $\sqrt T e^{i\sqrt T \tilde B_s} ds $ by the stochastic integration against
 $2 i e^{i\sqrt T \tilde B_s} d\tilde B_s$; see Proposition \ref{P3.1} below for a precise statement.
Now, in the following theorem, which might be interesting in itself, we prove the convergence in law  for the resulting martingales. 
 
\begin{theorem}\label{T2}
Let $(\tilde B_t)_{t\geq 0}$ be a standard  Brownian motion on $\R$. 
Let $g \in \mathcal B_b([0,1], \R^K)$ and $\ell_1,\ell_2 \in \mathcal B_b([0,1], \R^L)$. Then 
\begin{align*}
& \begin{pmatrix}\tilde B_t\\ 
\int_0^t  g(s)e^{i\sqrt T \tilde B_s} d\tilde B_s \\
\int_{s=0}^t  \left(\int_{r=0}^s \ell_2(r) \left( e^{i\sqrt T \tilde B_r} \right)   d \tilde B_r\right) \,  \ell_1(s) \left(e^{i\sqrt T \tilde B_s}\right)^* d \tilde B_s
\end{pmatrix}_{0\le t\le 1}
\\
\xrightarrow[T\to \infty]{Law}
&
\begin{pmatrix}
\tilde B_t\\
\frac{1}{\sqrt 2}\int_0^t  g(s) d\SW_s\\
\frac 1 2 
\int_{s=0}^t  \left(\int_{r=0}^s \ell_2 (r) d\SW_r\right)     \ell_1(s) \left(  d\SW_s \right)^*
\end{pmatrix}_{ 0\le t\le 1}
\end{align*}
where $(\SW_t)_{t\geq 0}$ is a standard  Brownian motion in $\mathbb C$ independent of $ (\tilde B_t)_{t\geq 0}$.
\end{theorem}
 
 The convergence in law in   Theorem \ref{T2} and in Theorem \ref{T1} below is with respect to the standard  uniform topology on 
 {$\mathcal C([0,1], \R^{1}\times \CC^K \times (\CC \otimes \CC^*)^L)$ where  $\CC \otimes \CC^* \simeq M_{2\times 1}(\R) \otimes M_{1\times 2}(\R)^* \simeq M_{2\times 2}(\R)$}.

We can now give the result for the additive functionals of the Brownian motion.
 
\begin{theorem}\label{T1}
Let $(\tilde B_t)_{t\geq 0}$ be a standard  Brownian motion on $\R$. 
Let $g \in \mathcal C^{1}([0,1], \R^K)$ and $\ell_1,\ell_2 \in \mathcal C^{1}([0,1], \R^L)$. Then 
\begin{equation}\label{ET1}
\begin{split}
& \begin{pmatrix}\tilde B_t\\ 
 \sqrt T \int_0^t  g(s)e^{i\sqrt T \tilde B_s} ds \\
 T \int_{s=0}^t  \left(\int_{r=0}^s \ell_2(r) e^{i\sqrt T \tilde B_r} dr\right)     \ell_1(s) \left(e^{i\sqrt T \tilde B_s}\right)^* ds
\end{pmatrix}_{0\le t\le 1}
\\
\xrightarrow[T\to \infty]{Law}
&
\begin{pmatrix}
\tilde B_t\\
 i\sqrt 2\int_0^t  g(s) d\SW_s\\
2 \int_{s=0}^t  \left(\int_{r=0}^s \ell_2(r)   \left( i d\SW_r \right)\right)     \ell_1(s) \left(  i\circ  d\SW_s \right)^*
\end{pmatrix}_{ 0\le t\le 1}
\end{split}
\end{equation}
where $(\SW_t)_{t\geq 0}$ is a standard  Brownian motion in $\mathbb C$ independent of $ (\tilde B_t)_{t\geq 0}$, and $\circ d\SW_s$ stands for Stratonovich integral.
\end{theorem} 
We chose to write Theorems \ref{T2} and \ref{T1} in this order since the 
 $\CC$-Brownian motion $(\SW_t)_{t\geq 0}$ in Theorem \ref{T1} may be taken to be the same as in Theorem \ref{T2}.

We now provide the 
proposition to transform additive functionals into martingales. 
\begin{proposition}
\label{P3.1}
Let $(\tilde B_t)_{t\geq 0}$ be a standard 1-dimensional Brownian motion and let $0\leq t \leq 1$. 

Let  $g$ be a $\cC^1$ function on $[0,1]$.
Then 
\begin{equation}\label{eq:sqrt-dt}
\sqrt T \int_0^t g(s) e^{i\sqrt T \tilde B_s} ds  = 2i \int_0^t g(s) e^{i\sqrt T \tilde B_s} d\tilde B_s + {R_{T,g}(t)}
\end{equation}
where the remainder term  $R_{T,g}(t)= \frac{-2g(t)}{\sqrt T} e^{i\sqrt T    \tilde B_t} + \frac{2g(0)}{ \sqrt T} e^{i\sqrt T \tilde B_0}  +  \frac{2}{\sqrt T} \int_0^t g'(s) e^{i\sqrt T \tilde B_s} ds $ satisfies:  
$$
\|R_{T,g}\|_\infty\le  \frac{1}{\sqrt T}\left(4\|g\|_\infty +2\|g'\|_\infty\right).
$$
\end{proposition}

 \begin{proof}[Proof of Proposition \ref{P3.1}]
Let $\psi_T$ be the function  $\psi_T(x)= \frac{-2}{T} e^{i\sqrt T x}$. It satisfies
\[
\psi_T'(x)= \frac{-2i}{\sqrt T} e^{i\sqrt T x},  \;  \psi_T''(x)= 2 e^{i\sqrt T x}
\]
and thus, by It\^o formula:

\begin{align*}
&\int_0^t g(s) e^{i\sqrt T \tilde B_s} ds=\frac{1} {2} \int_0^t g(s) \psi_T''( \tilde B_s) ds\\
   &= \int_0^t g(s) \Big( d\big(\psi_T( \tilde B_s)\big) - \psi_T'( \tilde B_s) d\tilde B_s \Big)\\
   &=  g(t) \psi_T( \tilde B_t) - g(0) \psi_T( \tilde B_0)  - \int_0^t g'(s) \psi_T( \tilde B_s) ds    -  \int_0^t g(s) \psi_T'( \tilde B_s) d\tilde B_s\\
    &= \frac{-2g(t)}{T} e^{i\sqrt T  \tilde B_t} + \frac{2g(0)}{T} e^{i\sqrt T  \tilde B_0}  +  \frac{2}{T} \int_0^t g'(s) e^{i\sqrt T \tilde B_s} ds     + \frac{2i}{\sqrt T}\int_0^t g(s) e^{i\sqrt T \tilde B_s} d\tilde B_s.
\end{align*}
This immediately yields~\eqref{eq:sqrt-dt}. The estimate of the remainder term $R_{T,g}$ is immediate.
\end{proof}

We now provide the proof of  the strong law for the oscillating integrals.\footnote{This nice and simple proof was communicated to us by one of the referees. We wish to thank her/him warmly. In an earlier version of the paper, we used a proof based on the occupation time formula for the Brownian motion.}
\begin{proof}[Proof of Theorem \ref{T:LFGN}]
Working coordinate-wise, it is enough to consider the case $K=1$. 
As for the standard Riemann-Lebesgue Lemma, by density of $\cC^1([0,1],\R)$ in $L^1([0,1],\R)$, it is enough to  treat  only the case where $g\in \cC^1([0,1],\R)$. Using Proposition \ref{P3.1}, one has for all $0\leq t\leq 1$,
\begin{align*}
I_{g,T}(t):&=\int_0^t g(s) e^{i\sqrt T \tilde B_s} ds\\
&= \frac{2i}{\sqrt T } \int_0^t g(s) e^{i\sqrt T \tilde B_s} d\tilde B_s + O\left( \frac{1} {T} \right)
\end{align*}
where the remainder term $O\left( \frac{1} {T} \right)$ is uniform in $t\in[0,1]$.
By the Burkholder-Davis-Gundy inequality, one has for $T\geq 1$:
\[
\E\left[ \sup_{t\in [0,1]}  |I_{g,T}|^8 \right] \leq \frac{C}{T^4}.  
\]
By Borell-Cantelli, this ensures the  desired uniform convergence for the subsequence
$(\sqrt n) _{n\geq 1}$:
\[
I_{g,\sqrt n} \xrightarrow[n\to \infty]{\Vert \cdot \Vert_\infty,\,  a.s.} 0.
\]
To conclude for a generic $T$, one uses that for $\sqrt n \leq  T <\sqrt{n+ 
1}$ and since the complex exponential is 1-Lipschitz,
\[
\Vert I_{g,T} -I_{g,\sqrt n} \Vert_\infty \leq |T-\sqrt n| \int_0^1 |g(s)| |\tilde B_s| ds 
\xrightarrow[n\to \infty]{ a.s.} 0.
\]
\end{proof}

\color{black}

Complementary to Proposition~\ref{P3.1}, the next lemma yields uniform bounds for the even moments of~\eqref{eq:sqrt-dt}.

\begin{lemma}
\label{L3.2}
Let $k\geq 1$. For any $\lambda >0$ {, $p\in\mathbb N$}  and any   bounded measurable deterministic function $g$ on $[0,1]^k$ with value in $\CC$, we have
\begin{equation}
    \label{E3.2bis}
     \E\left[\left|\int_{[0,1]^k} g(s) \prod_{j=1}^k e^{i\lambda \tilde B_{s_j}} ds_1 \dots ds_k \right|^{2p} \right]\le
\|g\|_\infty^{2p}\; \E\left[\left|\int_{[0,1]^k}  \prod_{j=1}^k e^{i\lambda \tilde B_{s_j}} ds_1 \dots ds_k \right|^{2p} \right].
\end{equation}
In particular, 
for all  $p\in \N$, there exists a constant  $C_p>0$ such that for all {$T>0$} and all bounded measurable functions $g$ on $[0,1]^k$ with value in $\CC$, 
\begin{equation}
    \label{E3.2}
   \E\left[\left| T^{\frac{k}{2}}   \int_{[0,1]^k} g(s) \prod_{j=1}^k e^{i\sqrt T \tilde B_{s_j}} ds_1 \dots ds_k \right|^{2p} \right] \le C_p\|g\|_\infty^{2p}.
\end{equation}
\end{lemma}

\begin{proof}
For simplicity, the proof is detailed only for $k=1$. 
The method   for $k\geq 2$ is similar  { replacing  all the one-dimensional integrals below on $[0,1]$ by multiple integrals on $[0,1]^k$. }
We compute 

\begin{align*}
    &\E\left[\left(\left|\int_0^1g(s)e^{i\lambda \tilde B_s}ds\right|^2\right)^p\right]= \E\left[
    \left(\int_{[0,1]^2} ds_1ds_2\; g(s_1)g(s_2) e^{i\lambda (\tilde B_{s_2}-\tilde B_{s_1})}
    \right)^p
    \right]\\
    &=\int_{[0,1]^{2p}}ds_1\ldots ds_{2p}\;g(s_1)\ldots g(s_{2p})\E\left[e^{i\lambda \sum_{k=1}^{2p}(-1)^k\tilde B_{s_k}}\right].
\end{align*}
Remarking that $\lambda \sum_{k=1}^{2p}(-1)^k\tilde B_{s_k}$ is a centered Gaussian random variable, we get that $\E\left[e^{i\lambda \sum_{k=1}^{2p}(-1)^k\tilde B_{s_k}}\right]$ is a positive number. Consequently, the last term can be bounded above by 
\[
\|g\|_\infty^{2p}\int_{[0,1]^{2p}}ds_1\ldots ds_{2p}\E\left[e^{i\lambda \sum_{k=1}^{2p}(-1)^k\tilde B_{s_k}}\right].
\]
But with the same calculation we get that this term is equal to
\[
\|g\|_\infty^{2p}\E\left[\left(\left|\int_0^1e^{i\lambda \tilde B_s}ds\right|^2\right)^p\right],
\]
yielding~\eqref{E3.2bis}.
Now, as in the proof  of Proposition~\ref{P3.1}, one has
\[
\int_0^1e^{i\lambda \tilde B_s}ds = 
\frac2{\lambda^2}\left(1-e^{i\lambda {\tilde B_1}
}\right)
+\frac{2i}\lambda \int_0^1e^{i\lambda \tilde B_s}d\tilde B_s
.
\]
As a consequence,  we get 
\begin{align*}
    \E\left[\left(\left|\lambda\int_0^1g(s)e^{i\lambda \tilde B_s}ds\right|^2\right)^p\right]&\le 2^{{p}}\lambda^{2p}\|g\|_\infty^{2p}
    \E\left[\left(
\frac{16}{\lambda^4}
+\frac{4}{\lambda^2} \left|\int_0^1e^{i\lambda \tilde B_s}d\tilde B_s
\right|^2\right)^p\right]\\
&\le 2^{{2p-1}}\lambda^{2p}\|g\|_\infty^{2p}\left(
\left(\frac{16}{\lambda^4}\right)^p+\frac{2^{2p}}{\lambda^{2p}}
\E\left[\left|\int_0^1e^{i\lambda \tilde B_s}d\tilde B_s
\right|^{2p}\right]\right)\\
&\le 2^{{2p-1}}\lambda^{2p}\|g\|_\infty^{2p}\left(
\left(\frac{16}{\lambda^4}\right)^p+\frac{2^{3p}}{\lambda^{2p}}
\E\left[\left|\beta_1
\right|^{2p}\right]
\right)\\
&\le C_p\|g\|_\infty^{2p}
\end{align*}
{where $\beta$ is a standard Brownian motion on $\mathbb{R}$.}
\end{proof}

{We can now give the proof of Theorem \ref{T2}.}

{To simplify the proof and the computation of the quadratic variations, in the proof we identify 
$\R^{1}\times \CC^K \times (\CC \otimes \CC^*)^L$ with $\R^{1+2K+4L}$.
Here we recall that $\CC\otimes\CC^* \simeq M_{2\times2}(\R)$ and we choose to identify 
 a real 2 by 2 matrix $M$ with  $(M_{1,1}, M_{1,2}, M_{2,1}, M_{2,2})^* \in \R^4$.}

\begin{proof}[Proof of Theorem \ref{T2}]

We will proceed in two steps: first we will consider only the vector valued martingale
\begin{equation}\label{E:Mg}
    (M_t^g(T))_{0\le t\le 1}:= \begin{pmatrix}\tilde B_t\\
\int_0^t g(s) e^{i\sqrt T \tilde B_s} d\tilde B_s
\end{pmatrix}_{0\le t\le 1} .
\end{equation}

Our goal is to show 
that the law of this martingale
converges to  (the law of)

\begin{equation}\label{E:Wg} 
\begin{pmatrix}
\tilde B_t\\
\frac{1}{\sqrt 2}  \int_0^t g(s) d\SW_s
\end{pmatrix}_{{0\le t\le 1}}.
\end{equation}
where $(\SW_t)_{t\geq 0}$ is a standard  Brownian motion in $\mathbb C$ independent of $ (\tilde B_t)_{t\geq 0}$.
For this, we will prove that the quadratic variation converges in law 
to the quadratic variation of the limit with respect to the standard uniform topology in $\mathcal C([0,T],  M_{d\times d}(\R))$ with $d=2K+1$.
We will in fact, using Theorem \ref{T:LFGN}, prove an almost sure convergence for the quadratic variation in this first step.
We will then use a simple criterion based on some uniform integrability of the quadratic variation  from Zheng \cite{Zheng:85} to prove that the martingale converges and  to identify the limit.

Let us denote by $\int_0^tu_s^g(T)ds$  the quadratic variation of the martingale $M_t^g$. The   $2K+1$ square symmetric  matrix $u_s^g$ is thus given by 
 
\begin{equation}\label{Eug}
u_s^g(T)= \begin{pmatrix}
 1& U_1^* & \dots & U_K^*\\
 U_1 & U_{1,1} & \dots & U_{1,K}\\
 \vdots & \vdots &   & \vdots\\
 U_K & U_{K,1} & \dots & U_{K,K}
\end{pmatrix}
\end{equation}
where for $1\leq k  \leq K $,
 \[
 U_k=  g_k(s) \begin{pmatrix}
  \cos(\sqrt T \tilde B_s) \\  \sin(\sqrt T \tilde B_s)
 \end{pmatrix}  
 \]
and for $1\leq  j,k \leq K$,
\[
 U_{j,k}= g_j(s) g_k(s) 
 \begin{pmatrix}
   \cos(\sqrt T \tilde B_s)^2 &  \cos(\sqrt T \tilde B_s)  \sin(\sqrt T \tilde B_s)\\
 \cos(\sqrt T \tilde B_s)  \sin(\sqrt T \tilde B_s)& \sin(\sqrt T \tilde B_s)^2
 \end{pmatrix}.
 \]

Using
\begin{align}
\nonumber & \begin{pmatrix}
   \cos(\sqrt T \tilde B_s)^2 &  \cos(\sqrt T \tilde B_s)  \sin(\sqrt T \tilde B_s)\\
 \cos(\sqrt T \tilde B_s)  \sin(\sqrt T \tilde B_s)& \sin(\sqrt T \tilde B_s)^2
 \end{pmatrix}\\
 = &\begin{pmatrix}
  \frac{1}{2} &0 \\
  0 & \frac{1}{2}
 \end{pmatrix}
+\begin{pmatrix}
\frac12\cos(2\sqrt{T} \tilde B_s)&\frac12\sin(2\sqrt{T} \tilde B_s)\\
\frac12\sin(2\sqrt{T} \tilde B_s)&-\frac12\cos(2\sqrt{T} \tilde B_s)
\end{pmatrix}
\label{E1/2}
\end{align}
and Theorem \ref{T:LFGN},
we obtain the following  almost sure convergence  of the quadratic variation of $M_t^g$:
\begin{equation}
    \label{E3.5bis}
    \sup_{t\in[0,1]}\left\|\int_0^tu_s^g(T)ds- A(t)\right\|\xrightarrow[T\to \infty]{a.s.} 0
\end{equation}
where $A(t)$ the $2K+1$ deterministic matrix  given by:
\begin{equation}\label{E:MatA}
A(t)= \begin{pmatrix}
 t & 0^* & \dots & 0^*\\
 0 & A_{1,1} & \dots & A_{1,K}\\
 \vdots & \vdots &   & \vdots\\
 0 & A_{K,1} & \dots & A_{K,K}
\end{pmatrix}
\end{equation}

where for $1\leq  j,k \leq K$,
\[
 A_{j,k}=  
 \begin{pmatrix}
  \frac 1 2 \int_0^t g_j(s) g_k(s) ds  &  0\\
 0 & \frac 1 2 \int_0^t g_j(s) g_k(s) ds 
 \end{pmatrix}.
 \]
 We recognize that 
$(A(t))_{t\in [0,1]}$ is the quadratic variation of the   martingale in~\eqref{E:Wg}.
{In particular, in view of~\eqref{E3.5bis},
the martingale in~\eqref{E:Wg} is the only possible limit in law for the family of  martingales~\eqref{E:Mg}.}

To show the existence of the limit (in law) of the martingale \eqref{E:Mg} as $T\to +\infty$ and the identification of the limit as (the law) of the martingale \eqref{E:Wg}, 
{it is thus enough to show that the family of martingales are tight for the topology of the uniform convergence on $\cC([0,1],\R^{1+2K})$.
}
According to Theorem~3 in \cite{Zheng:85}, 
it is enough to check that there exists $p>1$ such that the random variables   $\int_0^1\|u_s(T)\|^pds$ are {tight in $\R$, also called uniformly bounded in probability in \cite{Zheng:85}}; that is 
 \begin{equation} \label{E:UBP}
 \sup_{T\geq 1} \PP \left( \int_0^1\|u_s^g(T)\|^pds \geq M\right)  \xrightarrow[M\to \infty]{} 0.
 \end{equation}

The notation $\|u_s^g(T)\|(\omega)$ is for the norm or the  matrix $u_s^g(T)(\omega)$.
But here all entries of $u_s^g(T)(\omega)$ are uniformly bounded in $T$   and $\omega$.

This ends the first step of the proof.

In a second step, we add the iterated integrals and consider the vector valued martingale: 
\begin{equation} \label{E:Mgh}
 (M_t^{g,\ell}(T))_{0\le t\le 1}:=\begin{pmatrix}\tilde B_t\\ 
\int_0^t  g(s)e^{i\sqrt T \tilde B_s} d\tilde B_s \\
\int_{s=0}^t  \left(\int_{r=0}^s \ell_2(r) \left( e^{i\sqrt T \tilde B_r} \right)   d\tilde B_r\right) \,  \ell_1(s) \left(e^{i\sqrt T \tilde B_s}\right)^* d \tilde B_s
\end{pmatrix}_{0\le t\le 1}
.\end{equation}

The method of the proof is similar as in the first step. First we show that the quadratic variation converges in law. This will follow mainly from the first step of the proof.
Then we check the above uniform integrability of the quadratic variation to prove that the martingale converges and  to identify its  limit.

Let us denote by $\int_0^tu_s^{g,\ell}(T)ds$ the  quadratic variation of the martingale \eqref{E:Mgh}. It is given by the $1+2K +4L$  matrix:
\[
u_s^{g,\ell}(T)= \begin{pmatrix}
u_s^g(T) &  {v_s^{g,\ell}(T)}^*  \\
v_s^{g,\ell}(T)& w_s^{\ell,\ell}(T)
\end{pmatrix}
\]
where $u_s^g(T)$ is given by \eqref{Eug} and 
\[
w_s^{\ell,\ell}(T)=
\begin{pmatrix}
 W_{1,1}& \dots & W_{1,L}\\
 \vdots &  & \vdots\\
 W_{L,1}& \dots & W_{L,L}
\end{pmatrix}
\]
where for $1\leq j,l \leq L$
\[
W_{j,l}= 
  \ell_1^j \ell_1^l  \begin{pmatrix} 
J_{c,j} J_{c,l}   \theta_{1}^2  &  J_{c,j} J_{c,l}  \theta_{1} \theta_{2} & J_{c,j} J_{s,l}   \theta_{1}^2 & J_{c,j} J_{s,l} \theta_{1} \theta_{2} \\
J_{c,j} J_{c,l}  \theta_{1} \theta_{2}  &  J_{c,j} J_{c,l}  \theta_{2}^2  &  J_{c,j} J_{s,l}  \theta_{1} \theta_{2}& J_{c,j} J_{s,l}  \theta_{2}^2  \\
J_{s,j} J_{c,l}  \theta_{1}^2  & J_{s,j} J_{c,l}  \theta_{1} \theta_{2} & J_{s,j} J_{s,l}  \theta_{1}^2  & J_{s,j} J_{s,l}  \theta_{1} \theta_{2} \\
J_{s,j} J_{c,l}  \theta_{1} \theta_2  & J_{s,j} J_{c,l}  \theta_{2}^2  & J_{s,j} J_{s,l}  \theta_{1} \theta_2  & J_{s,j} J_{s,l} \theta_{2}^2  
                 \end{pmatrix}
\]
with for $1\leq l \leq L$
\[
J_{c,l}(s)= \int_{r=0}^s \ell_2^l(r)  \cos({\sqrt T \tilde B_r}) d \tilde B_r, \,
J_{s,l}(s)= \int_{r=0}^s \ell_2^l(r)  \sin({\sqrt T \tilde B_r}) d \tilde B_r 
\]
and 
\[
 \theta_{1}(s)=  \cos({\sqrt T \tilde B_s}), \, \theta_{2}(s)=  \sin({\sqrt T \tilde B_s});
 \]
 and where 
 \[
v_s^{g,\ell}(T) =
\begin{pmatrix}
 V_1\\
 \vdots\\
 V_L
\end{pmatrix}
\]
where for $1\leq l \leq L$ 
\[
V_l= \ell_1 ^l \begin{pmatrix}
 J_{c,l} \theta_1 &  J_{c,l}  g^1 \theta_1^2 & J_{c,l}  g^1 \theta_1 \theta_2 & \dots  &  J_{c,l}  g^K \theta_1^2 & J_{c,l}  g^K \theta_1 \theta_2\\
 J_{c,l} \theta_2 &  J_{c,l}  g^1 \theta_1 \theta_2 & J_{c,l}  g^1  \theta_2^2 & \dots  &  J_{c,l}  g^K \theta_1 \theta_2 & J_{c,l}  g^K  \theta_2^2\\
  J_{s,l} \theta_1 &  J_{s,l}  g^1 \theta_1^2 & J_{s,l}  g^1 \theta_1 \theta_2 & \dots  &  J_{s,l}  g^K \theta_1^2 & J_{s,l}  g^K \theta_1 \theta_2\\
   J_{s,l} \theta_2 &  J_{s,l}  g^1 \theta_1 \theta_2 & J_{s,l}  g^1  \theta_2^2 & \dots  &  J_{s,l}  g^K \theta_1 \theta_2 & J_{s,l}  g^K  \theta_2^2\\
 \end{pmatrix}.
\]

By the first part of the proof,  we deduce that there exists $(\SW_t)_{0\leq t \leq 1}$ a Brownian motion independent of $(\tilde B_t)_{0\leq t \leq 1}$ such that 
\[
\begin{pmatrix}
 \tilde B_t \\
 \int_0^t g(s) e^{i\sqrt T \tilde B_s} d\tilde B_s\\
 \int_0^t \ell_2(s) e^{i\sqrt T \tilde B_s} d\tilde B_s\\
\end{pmatrix}_{0\leq t\leq 1}
\xrightarrow[T\to \infty]{Law}
\begin{pmatrix}
 \tilde B_t \\
 \frac{1} { \sqrt 2 }\int_0^t g(s) d\SW_s\\
 \frac{1} { \sqrt 2 }\int_0^t \ell_2(s) d\SW_s
\end{pmatrix}_{0\leq t\leq 1}.
\]

Using again \eqref{E1/2},
we prove that the quadratic variation of 
$\int_0^t u_s^{g,\ell}(T) ds$ of $M^{g,\ell}$ converges in law to:
\[
\int_0^t u_s^{g,\ell}(T) ds \xrightarrow[T\to\infty]{Law}  \int_0^t
\begin{pmatrix}
 A(s)  &  B(s)^*  \\
B(s) & C(s)
\end{pmatrix} ds
\]

\[
B(s) =
\begin{pmatrix}
 B_1\\
 \vdots\\
 B_L
\end{pmatrix}
\]
where for $1\leq l \leq L$ 
\[
B_l= \frac{1}{2\sqrt 2} \ell_1^l(s) \begin{pmatrix}
 0 &  I_{1,l}  g^1  Id_2 &   \dots  &   I_{1,l}  g^K  Id_2\\
 0 &  I_{2,l}  g^1  Id_2 &   \dots  &   I_{2,l}  g^K  Id_2
 \end{pmatrix},
\]

\[
C(s)=
\begin{pmatrix}
 C_{1,1}& \dots & C_{1,L}\\
 \vdots &  & \vdots\\
 C_{L,1}& \dots & C_{L,L}
 \end{pmatrix},
C_{j,l}= 
 \frac{1}{4} \ell_1^j(s) \ell_1^l(s)  \begin{pmatrix} 
 (I_{1,j} I_{1,l}) Id_2  &  (I_{1,j} I_{2,l}) Id_2\\
 (I_{2,j} I_{1,l}) Id_2  &  (I_{2,j} I_{2,l}) Id_2                  \end{pmatrix}
\]
with 
\[
I_{1,j}= \int_0^s \ell_2^j (r) d\SW_r^1,\; I_{2,j}= \int_0^s \ell_2^j (r) d\SW_r^2.
\]

{To be complete and deal with remainder terms,  we must show for example that 
\[
\E\left[ \left(\int_0^t \Big(\int_0^s \ell_2(r) \cos(\sqrt T \tilde B_r ) d\tilde B_r\Big)^2 \; \frac 1 2 \sin (2\sqrt T \tilde B_s)     ds  \right)^2\right] \xrightarrow[T\to \infty]{}0.
\]
}
 This is direct using Itô formula as in Proposition \ref{P3.1}, replacing $g$ by the martingale $\displaystyle M_s=\int_0^s \ell_2(r) \cos(\sqrt T \tilde B_r ) d\tilde B_r$:
\begin{align*}
\int_0^tM_s^2\frac{\sin(2\sqrt{T}\tilde B_s)}{2}ds&=-M_t^2\frac{\sin(2\sqrt{T}\tilde B_t)}{4T}+\int_0^t M_s\frac{\sin(2\sqrt{T}\tilde B_s)}{2T}dM_s\\
&+\int_0^t\frac{\sin(2\sqrt{T}\tilde B_s)}{4T}(dM_s,dM_s)+\int_0^t M_s^2\frac{\cos(2\sqrt{T}\tilde B_s)}{2\sqrt{T}}d\tilde B_s\\
&+\int_0^t\frac{\cos(2\sqrt{T}\tilde B_s)}{\sqrt T}M_s(dM_s,d\tilde B_s).
\end{align*}

 The limit of the quadratic variation corresponds to the quadratic variation of the vector valued martingale:
 \begin{equation}\label{E:Wgh}
\begin{pmatrix}
\tilde B_t\\
\frac{1}{\sqrt 2}\int_0^t  g(s) d\SW_s\\
\frac{1}{2}
\int_{s=0}^t  \left(\int_{r=0}^s \ell_2(r) d\SW_r\right)     \ell_1(s) \left(  d\SW_s \right)^*
\end{pmatrix}_{ 0\le t\le 1}.
\end{equation}

We now turn to the uniform integrability condition \eqref{E:UBP}. We take $p=2$ 
and show the stronger fact that all the entries of the matrix $u_s^{g,\ell}(T)$ are uniformly bounded in $L^2(\Omega)$ for  $T\geq 1$ and $s\in[0,1]$. 
This  is clear for the entries of $v_s^{g,\ell}$ using Itô isometry.
 This is  also the case for the entries of $w_s^{\ell,\ell}$  since for example 
\[
\E\left[  \left(\int_0^1 \ell_2 (r) \cos(\sqrt T \tilde B_r) d\tilde B_r\right)^4\right]
\leq E[G^4]
\]
for $G$ a centered gaussian variable with variance $\Vert \ell_2 \Vert_\infty^2$.
Similarly  as in the first step, this gives the existence of the limit in law of the martingale \eqref{E:Mgh} and its identification as the law of  the martingale \eqref{E:Wgh}.
\end{proof}

We now turn to the proof of Theorem \ref{T1}.
\begin{proof}[Proof of Theorem \ref{T1}]
First using Proposition \ref{P3.1}, we have
\[
\sqrt T \int_0^t g(s) e^{i\sqrt T \tilde  B_s} ds  = 2i \int_0^t g(s) e^{i\sqrt T \tilde B_s} d \tilde B_s + R_{T,g}(t)
\]
with $R_{T,g}$ converging a.s. to 0 uniformly on $[0,1]$. 
Using also Proposition \ref{P3.1} with the $\mathcal C^1$ function: $\ell(s)=\ell_1(s) \int_{r=0}^s \ell_2(r) \sqrt T e^{i\sqrt T \tilde B_r} dr$, we have:
\begin{align}
 \nonumber J_T(t):=& \int_{s=0}^t \ell(s) \cdot \left(\sqrt T e^{i\sqrt T \tilde B_s}\right)^* ds\\
\nonumber =& - \frac{2}{T}  \left( \ell_1(t) \int_{r=0}^t \ell_2(r) \sqrt T e^{i\sqrt T \tilde B_r} dr \right)
\cdot   \left(\sqrt T e^{i\sqrt T \tilde B_t}\right)^*\\
\nonumber & + \frac{2}{T} \int_{s=0}^t
  \ell_1'(s) \int_{r=0}^s \ell_2(r) \sqrt T e^{i\sqrt T \tilde B_r} dr \cdot \left( \sqrt T e^{i\sqrt T \tilde B_s}\right)^* ds  \\
 \label{E:3.18} & + \frac{2}{T} \int_{s=0}^t
  \ell_1(s) \ell_2(s) \sqrt T e^{i\sqrt T \tilde B_s}  \cdot \left( \sqrt T e^{i\sqrt T \tilde B_s}\right)^* ds \\
 \label{E:3.19}&+ 2  \int_{s=0}^t \left(\ell_1(s) 
                 \int_{r=0}^s \ell_2(r) \sqrt T e^{i\sqrt T \tilde B_r} dr \right) \cdot  \left(  {i}   e^{i\sqrt T \tilde B_s} \right)^* d\tilde B_s
\end{align} 
Using Lemma \ref{L3.2}, we directly see that, when $T\to \infty$, the first two lines in $J_T(t)$ converges (uniformly in $t\in [0,1])$) to 0 in $L^2(\Omega)$.
By Theorem \ref{T:LFGN} and using \eqref{E1/2}, the term \eqref{E:3.18} on the third line converges to 
\begin{equation} \label{E:g-strato}
\int_{s=0}^t
  \ell_1(s) \ell_2(s) \begin{pmatrix}
   1&0\\0&1
  \end{pmatrix} ds.
\end{equation}

Using again Proposition \ref{P3.1} for  the last term \eqref{E:3.19}, one has
\begin{align}
  2  \int_{s=0}^t &\left(\ell_1(s) 
                 \int_{r=0}^s \ell_2(r) \sqrt T e^{i\sqrt T \tilde B_r} dr \right)  \cdot  \left(  {i}   e^{i\sqrt T \tilde B_s} \right)^* d\tilde B_s\notag \\
=& 4  \int_{s=0}^t \left(\ell_1(s) 
                 \int_{r=0}^s \ell_2(r) \left(i e^{i\sqrt T \tilde B_r}\right) d\tilde B_r \right) \cdot  \left(  {i}   e^{i\sqrt T \tilde B_s} \right)^* d\tilde B_s \label{Emart1} \\
+& 2  \int_{s=0}^t  \ell_1(s) 
                 R_{T,\ell_2}(s) \cdot  \left(  {i}   e^{i\sqrt T \tilde B_s} \right)^* d\tilde B_s\label{Emart2}
                 \end{align}
In particular \eqref{Emart2} also tends to 0. Finally, 
by Theorem \ref{T2}, the term \eqref{Emart1} converges  in law to 
\begin{equation} \label{Elim}
2 \int_{s=0}^t \left(\int_{r=0}^s \ell_2 (r)  (i d\SW_r)\right)     \ell_1(s) \left( i d\SW_s \right)^*.
\end{equation}

{
The proof is finished by using Slutsky's theorem and since the sum of \eqref{E:g-strato} and \eqref{Elim} provides the announced Stratonovich integral.}

\end{proof}

\section{Integrability and convergence of the dual}

{In all this section, we study the long-time behaviour and the asymptotics of the Malliavin dual. The estimates obtained here will be used in the next section to deduce gradient estimates of the semigroup.}

\subsection{Asymptotic behaviour of the Malliavin matrix}

{We focus first  on the Malliavin matrix. We introduce here a natural scaling   and  prove the  convergence in law  of the the Malliavin matrix under this rescaling.}

Note that  in  $L^2([0,T])$, for $i,j \in\{0,1,2\}$,
\[
\langle h_i, h_j \rangle = \int_0^T h_i(t) h_j(t) dt
                                 =  T \int_0^1 h_i(rT) h_j(rT) dr
\]
and that for $0\leq r\leq 1$,
\[\left\{
\begin{array} {c c l}
 h_0(Tr) &=& \quad 1,\\
 h_1(Tr) &=& \quad  \int_0 ^{Tr} \sin(B_t) dt
            =\quad \sqrt T \left(\sqrt T  \int_0^r  \sin( \sqrt T \tilde B_{s}) ds\right),\\
 h_2(Tr) &=& - \int_0 ^{Tr} \cos(B_t) dt
            = \sqrt T \left(-\sqrt T  \int_0^r  \cos( \sqrt T \tilde B_{s}) ds\right),
\end{array} \right.
\]
where $(\tilde B_r)_{0\leq r \leq 1}$ is still defined in \eqref{E31bis}.

In view of  Theorem \ref{T1}, it is thus  natural to introduce the {rescaled} matrix $\tilde \cC(T)$ defined by:
\begin{equation}\label{EtildeC}
 \cC(T) = \begin{pmatrix}
 \sqrt T & 0 &0  \\
 0 &T&0 \\
 0&0&T\\
 \end{pmatrix}
 \tilde  \cC(T) \begin{pmatrix}
 \sqrt T & 0 &0  \\
 0 &T&0 \\
 0&0&T\\
 \end{pmatrix}.
\end{equation}

Note that $\tilde \cC(T)$ is also the Gram matrix but in $L^2([0,1])$ of $(\tilde h_0,\tilde  h_1,\tilde h_2)$ where for $0\leq r \leq 1$,
\begin{equation}
\label{E52}
\left\{
\begin{array} {c c l}
\tilde h_0(r)&=& 1 \\
\tilde h_1(r)&=&     \sqrt T \int_0^r \sin (\sqrt  T \tilde B_s) ds \\ 
\tilde h_2(r)&=& - \sqrt T \int_0^r \cos(
\sqrt T \tilde B_s) ds \end{array}\right..
\end{equation} In particular, we may also write
\begin{equation}
    \label{E53}
    \tilde \cC(T)=\int_0^1\tilde I_T(r)\tilde I_T(r)^{*} dr
\end{equation} with 
\[
    \tilde I_T(r):= 
    \begin{pmatrix}
 1 \\   
 \sqrt T \int_0^r \sin (\sqrt T \tilde B_s) ds \\ 
-  \sqrt T \int_0^r \cos (\sqrt T \tilde B_s) ds \end{pmatrix}
    =\begin{pmatrix}
 1 \\   -i\sqrt T \int_0^r e^{
 i\sqrt T \tilde B_s} ds \end{pmatrix}.
\]

We have the following convergence in law:
\begin{proposition}\label{PG}
One has the convergence:
\begin{equation}
\label{E1}
\tilde \cC(T) \xrightarrow[T\to \infty]{Law} \cG 
 \end{equation}
 with
 \begin{equation}
\label{E1bis}
\cG =\begin{pmatrix}
 1 & \sqrt2\int_0^1 {\SW}_s^1 ds & \sqrt2\int_0^1 {\SW}_s^2 ds \\
 \sqrt2\int_0^1 {\SW}_s^1 ds &2\int_0^1 ({\SW}_s^1)^2 ds&2\int_0^1 {\SW}_s^1 {\SW}_s^2 ds \\
 \sqrt2\int_0^1 {\SW}_s^2 ds& 2\int_0^1 {\SW}_s^1 {\SW}_s^2 ds &2\int_0^1 ({\SW}_s^2)^2 ds\\
 \end{pmatrix}
 \end{equation}
  where   $\displaystyle\SW=\SW^1+i\SW^2$ is the Brownian motion from Theorem \ref{T2} and Theorem \ref{T1}. 
\end{proposition}
We may also write:
\[
\cG=\int_0^1\tilde I_\infty(r)\tilde I_\infty(r)^{*} dr
\] with 
\begin{equation}\label{E1ter}
   \tilde I_\infty(r)=\begin{pmatrix}
   1\\\sqrt2 {\SW}_r^1\\\sqrt2{\SW}_r^2
   \end{pmatrix}=\begin{pmatrix}
   1\\\sqrt2 {\SW}_r
   \end{pmatrix}
\end{equation}
and with the identification $\displaystyle\SW=\binom{\SW^1}{\SW^2}$.

\begin{proof}[Proof of Proposition \ref{PG}]
 Observe that by \eqref{E53}
 the coefficients of the matrix $\tilde \cC(T)$ are obtained by products and integration of the left hand side of~\eqref{ET1} and thus   are a continuous functional of the left hand side of~\eqref{ET1}. Using Theorem \ref{T1}, we can conclude that $\tilde \cC(T)$ converges in law to this functional applied to the right hand side of~\eqref{ET1}, which yields~\eqref{E1bis}. 
\end{proof}

\subsection{Moments convergence of the renormalised matrix  $\tilde{\cC}(T)^{-1}$}
 
The next proposition is a key contribution for the main result of the paper. It brings quantitative asymptotic behaviour of  moments of $\tilde \cC(T)^{-1}$. Even if the existence of these  moments is well known by Malliavin calculus, obtaining uniform bounds in $T$ requires a careful study and additional tools.

\begin{proposition}
\label{P2}
For all $p> 0$, there exists $C_p>0$ such that for all $T\ge {\sqrt{2}}$,
\begin{equation}
\label{E3}
 \esp\left[\left\|\tilde \cC(T)^{-1}\right\|^{p}\right]\le C_p.
\end{equation}
Moreover
\begin{equation}
    \label{E2}
    \esp\left[\left\|\tilde \cC(T)^{-1}\right\|^{p}\right] \xrightarrow[T\to \infty]{}\esp\left[\left\| \cG^{-1}\right\|^{p}\right]
\end{equation}
where $\cG$ is defined by \eqref{E1}.
\end{proposition}
A direct consequence of the above proposition is that the right hand side of~\eqref{E2} is finite for all $p>0$. 

{As it is seen from Lemma \ref{L3}, an important step in the proof is to show that the oscillating integrals have sub-Gaussian densities. This is given in Lemmas \ref{L1} and \ref{LY} below. This comes from recent results from \cite{Bobkov:23} and  by showing 
 appropriate slicing of the oscillating integrals have sub-Gaussian densities. The proofs of Lemmas \ref{L3}, \ref{L1} and  \ref{LY} will be done in the next subsection.
 }

\begin{proof}[Proof of Proposition \ref{P2}]

To prove \eqref{E2}, the convergence in law of $\tilde\cC(T)$ to $\cG$ is a key point but it is not sufficient.
However it is enough to prove \eqref{E3}
{
with a parameter $p'$  larger than the parameter $p$ in~\eqref{E2}. Indeed, for fixed $p\in [1,p')$ this will prove uniform integrability of the random variables $\displaystyle \left\|\tilde \cC(T)^{-1}\right\|^{p}$ and convergence~\eqref{E2} for $p$ will be a consequence of Theorem~3.5 in \cite{Billingsley:99}.
}

We then turn to prove \eqref{E3}. 
 
Since moments of $\|\tilde\cC(T)\|$ are uniformly bounded in $T$ ({using Lemma~\ref{L3.2}}), Lemma~4.7 in \cite{Hairer:11} tells us that for proving~\eqref{E3}, it is enough to prove that for all $p>0$, there exists $C_p>0$ (the constant $C_p$ will change from one equation to another) such that for all column vector $a\in \R^3$ satisfying $\|a\|=1$, $\ve>0$, $T>{\sqrt{2}}$,
\begin{equation}
    \label{E4}
    \PP\left[a^{*} \tilde \cC(T) a\le \ve\right]\le C_p\ve^p.
\end{equation}
By Markov inequality,
 \begin{equation}
    \label{E6}
    \PP\left[a^{*} \tilde \cC(T) a\le \ve\right]= \PP\left[\left(a^{*} \tilde \cC(T) a\right)^{-1}\ge \ve^{-1}\right]\le \ve^p\esp\left[\left(a^{*} \tilde \cC(T) a\right)^{-p}\right].
\end{equation}
So we are left to get a uniform bound in $T\ge {\sqrt{2}}$ and unit vector $a\in \R^3$ for $\esp\left[\left(a^{*} \tilde \cC(T) a\right)^{-p}\right]$.

Since $\tilde \cC(T)$ is a Gram matrix, for any $a\in \R^3$, one has 
\begin{align*}
a^{*} \tilde \cC(T) a &= \sum_{i,j=0}^2 a_i \; \langle \tilde h_i, \tilde h_j\rangle_{L^2([0,1])} \; a_j\\
                       &=  \Big\Vert \sum_{i=0}^2 a_i \tilde h_i  \Big\Vert_{L^2([0,1])} ^2.
\end{align*}
\begin{enumerate}
    \item Let us firstly consider the case when $p\in (0,1/2)$.

   {Writing $a=\left(\begin{array}{c}
             \cos\varphi \\
             \sin\varphi \,\eta
        \end{array}\right)$ with unitary $\displaystyle \eta=e^{-i\theta}$ we have:}  
        
        \begin{equation}
            \label{E8}
       a^{*} \tilde \cC(T) a =    \frac1{T^2} \int_0^T  \left(\sqrt{T}\cos\varphi+\sin\varphi \int_0^t \sin \left(
       \theta
       + B_s\right) ds\right)^2 dt. 
        \end{equation}
      
        This yields the bound below for $T\ge \sqrt2$
        \[
      a^{*} \tilde \cC(T) a \ge \frac14
      \frac{2}{T^2-1}\int_1^T   \left(\sqrt{\frac{T}{t}}\cos\varphi+\sin\varphi \frac1{\sqrt t}\int_0^t \sin \left(
      \theta
      + B_s\right) ds\right)^2 t\,dt
        \]
        and then by convexity,
        \begin{equation}
            \label{E9}
            \begin{split}
          &  \esp\left[\left(a^{*} \tilde \cC(T) a\right)^{-p}\right]\\&\le 
     4^p     \frac{2}{T^2-1}\int_1^T    
     \esp\left[\left|\sqrt{\frac{T}{t}}\cos\varphi+\sin\varphi \frac1{\sqrt t}\int_0^t \sin \left(
     \theta
     + B_s\right) ds\right|^{-2p}\right] t\, dt
     \end{split}
        \end{equation}
      {and   the result follows from the following Lemmas \ref{L3} and \ref{L1}} {whose proofs are postponed to Subsection \ref{Subsec Lemmes}.}

        \begin{lemma}\label{L3}
        {Fix  $\alpha\in (0,1)$. Then
        \begin{equation}\label{E10bis}
         \sup_{\varphi\in[0,2\pi], a\geq 1, Y \in \mathcal {RV}_{M,\sigma}}\esp\left[\left|a \cos\varphi+\sin\varphi Y  \right|^{-\alpha}\right] 
         \leq C(M,\sigma,\alpha)<+\infty
         \end{equation}
         where $\mathcal {RV}_{M,\sigma}$ denotes the set of random variables $Y$ with sub-Gaussian density bounded by: 
         \[
         p_Y(x) \leq M \exp\left(-\frac{x^2}{2\sigma^2}\right).\]
         }
        \end{lemma}
        \begin{lemma}
\label{L1}
Denote for $\theta\in [0,2\pi]$, $t>0$, $\displaystyle J(\theta,t):= \frac1{\sqrt{t}}\int_0^t\sin(\theta+B_s) ds$. 
First we denote by $J(\theta,b,t)$, the random variable obtained by  conditioning $J(\theta,t)$ on $B_t=b \ {\rm mod}\ 2\pi$. Then  $J(\theta,b,t)$ has a density $x\mapsto p_{\theta,b,t}(x)$, and for all $t_0>0$ there exist $M(t_0)>0$, $\sigma(t_0)>0$ such that for all $\theta, b\in \R$ and $t\ge t_0$, 
\begin{equation}
    \label{E7}
    p_{\theta,b,t}(x)\le M(t_0)\exp\left\{-\frac{x^2}{\sigma(t_0)^2}\right\}.
\end{equation}
As a consequence,  $J(\theta,t)$ also has a density $x\mapsto p_{\theta,t}(x)$ which satisfies for all $\theta\in \R$ and all $t\geq t_0 $ 
\begin{equation}
    \label{E7bis}
    p_{\theta,t}(x)\le M(t_0)\exp\left\{-\frac{x^2}{\sigma(t_0)^2}\right\}.
\end{equation}
\end{lemma}

    \item Finally we consider the general case $p>0$.
  
     For general $p\geq 1/2$, let $N\geq 2$ such that $N>2p$. We write 
      \begin{equation}
        \label{E11bis}
        \begin{split}
        a^{*} \tilde \cC(T) a \ge &   \sum_{k=1}^N  \frac{1}{T^2} \int_{\frac{(k-1/2)T}{N}}^{k\frac{T}{N}}  \left(\sqrt{T}\cos\varphi+\sin\varphi \int_0^t \sin \left(\theta + B_s\right) ds\right)^2 dt. 
        \end{split}
    \end{equation}
Using the  arithmetic-geometric inequality we get 
    \begin{equation}
        \label{E12bis}
        a^{*} \tilde \cC(T) a  \ge   \frac{N}{T^2} \prod_{k=1}^N\left(\int_{(k-1/2)\frac{T}{N}}^{k\frac{T}{N}}  \left(\sqrt{T}\cos\varphi+\sin\varphi \int_0^t \sin \left(\theta+ B_s\right) ds\right)^2 dt\right)^{1/N} .
    \end{equation}
    
    {Therefore, 
    \begin{align*}
       & \esp\left[\left(a^{*} \tilde \cC(T) a\right)^{-p}\right]\\
       &\le N^p \esp\left[\prod_{k=1}^N\left(\frac{1}{T^2}\int_{(k-1/2)\frac{T}{N}}^{k\frac{T}{N}}  \left(\sqrt{T}\cos\varphi+\sin\varphi \int_0^t \sin \left(\theta + B_s\right) ds\right)^2 dt\right)^{-p/N} 
       \right]\\
       &= N^p \esp\left[\prod_{k=1}^N\left(\frac{1}{T^2}\int_{(k-1/2)\frac{T}{N}}^{k\frac{T}{N}}  \left(\frac{\sqrt{T}}{\sqrt{t_k}} \cos\varphi+\sin\varphi  \frac{1}{\sqrt t_k}\int_0^{t_k} \sin \left( \theta + B_s\right) ds\right)^2 t_k\, dt_k\right)^{-p/N}
       \right].
       \end{align*}
     By the Jensen inequality applied for each of the  probability measure 
     { $$\frac{2N^2}{T^2(k-\frac{1}{4})} t_k\mathbb{1}_{[(k-1/2)\frac{T}{N},k\frac{T}{N}]} dt_k$$} and the Fubini theorem, one obtains the (rough) estimate
      \begin{align*}
       & \esp\left[\left(a^{*} \tilde \cC(T) a\right)^{-p}\right]\\
       &\le N^p  \int_{\frac{T}{2N}}^{\frac{T}{N}} \dots \int_{(N-1/2)\frac{T}{N}}^{T}
       \left(\frac{1-\frac{1}{4}}{2N^2}\right)^{p/N-1} \frac{t_1}{T^2} dt_1 \dots \left(\frac{N-\frac{1}{4}}{2N^2}\right)^{p/N-1}\frac{t_N}{T^2} dt_N \\ 
    &\esp\left[\prod_{k=1}^N  \left(\frac{\sqrt{T}}{\sqrt{t_k}} \cos\varphi+\sin\varphi  \frac{1}{\sqrt t_k}\int_0^{t_k} \sin \left( \theta + B_s\right) ds\right)^{-2p/N} \right] \\
      &  \vspace{1cm}
      \le 2^{-p}  \int_{\frac{T}{2N}}^{\frac{T}{N}} \dots \int_{(N-1/2)\frac{T}{N}}^{T}
       \frac{2N^2}{T^2(1-\frac{1}{4})} t_1 dt_1 \dots \frac{2N^2}{T^2\left(N-\frac{1}{4}\right)}t_Ndt_N
    \esp\left[\prod_{k=1}^N  F_k(Y(t_k)) \right]
       \end{align*}
       with 
       \begin{equation}\label{EY}
       Y(t_k)=  \frac{1}{\sqrt t_k}\int_0^{t_k} \sin \left(\theta + B_s\right) ds 
       \end{equation}
       and
       \begin{equation}\label{EF}
       F_k(y)=\left(\frac{\sqrt{T}}{\sqrt{t_k}} \cos\varphi+\sin\varphi  y\right)^{-2p/N}.
       \end{equation}
       }
     By Lemma \ref{LY} below, 
      \[
      \esp\left[\prod\limits_{k=1}^N  F_k(Y(t_k)) \right]
    \leq C_N'.
      \]
      As a consequence, 
      \[ \sup_{a\in \R^3, \Vert a \Vert=1, T \geq 1}
      \esp\left[\left(a^{*} \tilde \cC(T) a\right)^{-p}\right]
      \leq C_p
      \]
      and Proposition \ref{P2} follows.
      \end{enumerate}
\end{proof}

\begin{lemma}\label{LY}
 Let $T \geq 1$, $N\geq 2$.  Then there exist some constants $C_N>0$ and $\sigma_N>0$ independent of $\theta \in  \R$   and of $(t_1,t_2\ldots,t_N)\in \prod\limits_{k=1}^N \left[\frac{(k-1/2)T}{N},\frac{kT}{N}\right]$ such that the density  of the vector $Y=(Y(t_1), \dots, Y(t_N))$ defined in \eqref{EY} is bounded for $y\in\R^N$ by:
 \begin{equation}\label{EfY}
f_Y(y) \leq C_N \exp \left(- \frac{ \Vert  y \Vert^2}{2\sigma_N^2} \right).
 \end{equation}
 As a consequence, there exists a constant $C_N'$ such that  for all $\theta \in \R$:
  \[
     \forall (t_1,t_2\ldots,t_N)\in\prod\limits_{k=1}^N \left[\frac{(k-1/2)T}{N},\frac{kT}{N}\right], \qquad\esp\left[\prod\limits_{k=1}^N  F_k(Y(t_k)) \right]
    \leq C_N' .
      \]
\end{lemma}

\subsection{Proofs of Lemmas \ref{L3}, \ref{L1} and \ref{LY}}\label{Subsec Lemmes}
{We now turn to the proofs of the three remaining Lemmas. We begin with showing that $\displaystyle J(\theta,t):= \frac1{\sqrt{t}}\int_0^t\sin(\theta+B_s) ds$ has a sub-Gaussian density, i.e. Lemma \ref{L1}.}

\begin{proof}[Proof of Lemma \ref{L1}]
We  prove the result for  $J(\theta,b,t)$ the conditioning of  $J(\theta,t)$ on $B_t=b$. The result for $J(\theta, t)$ follows by integrating  the conditional  density \eqref{E7}.

 The proof will be based on arguments of Bobkov \cite{Bobkov:23}. It will be separated in two arguments. First, we will prove that the density $p_{\theta,b,t}$ of $J(\theta,b,t)$ is bounded. Then in a second step, we will  prove the gaussian decay of the density. For simplicity we only make the proof for $t_0=4$.

To prove that the density is bounded,  we split $J$ in "a lot of small pieces":
for $t\geq \frac{t_0}{2}=2$, let $n=[t]$  and write
\[ 
J(\theta,t)=\frac1{\sqrt t} \int_0^t \sin \left( B_s\right) ds 
       =\frac {\sqrt n} {\sqrt t}   \left( \frac{X_1+\dots + X_n}{\sqrt n}\right)
 \]   
with 
\[
 X_k=\int_{k-1}^k \sin \left( B_s\right) ds \textrm{ for } 1\leq k \leq n-1 \textrm{ and } X_n= \int_{n-1}^t\sin \left(B_s\right) ds
\]
and where, here, $(B_s)$ is a real Brownian motion starting from $\theta$.
 We now condition on the event $(B_0,\dots,B_{n-1}, B_t)=(b_0,\dots, b_{n-1}, b_n)\mod 2\pi$ with $b_0=\theta, b_n=b$. Under this conditioning, the variables $(X_i)_{1\leq i\leq n}$ become independent.
  Moreover the random variables 
$\displaystyle
\int_0^r\sin(B_s^{a,b})ds 
$, with $a,b\in \bbS$, $r\in [1,2]$ and $B_s^{a,b}$ a Brownian bridge from $B_0^{a,b}=a$ to $B_r^{a,b}=b$, have densities uniformly bounded by some constant $M>0$.
Indeed, applying Hörmander Theorem to the kinetic Brownian motion, it is clear that the densities of
the random vectors $\left(B_r \mod(2\pi),\displaystyle
\int_0^r\sin(B_s)ds\right)
$  on $\bbS\times\mathbb{R}$ , for any starting points $a\in \bbS$, $r\in[1,2]$, are uniformly bounded. As the densities of $B_r\mod(2\pi)$   on $\bbS$ , for $a\in \bbS$, $r\in[1,2]$, are uniformly lower-bounded, the claim is obtained from conditional to $B_r=b\in \bbS$.

Consequently, 
by \cite{Bobkov:23} formula (1.7), $\displaystyle \frac1{\sqrt{n}}(X_1+\cdots+X_n)$ has a density bounded by $\sqrt{e}M$. 
Since the ratio $\frac{\sqrt n} {\sqrt t}$ is bounded from below, one infers that conditionaly on the event, $(B_0,\dots,B_{n-1}, B_t)=(b_0,\dots, b_{n-1}, b_n) \mod 2\pi$, $J(\theta,t)$ has also a bounded density and with a constant independent of the value of $(b_0,\dots b_{n})\mod 2\pi$.
By integrating this bound, one deduces that 
\begin{equation}\label{E:4.21}
p_{\theta,b,t}(x)\le M 
\end{equation}
for some constant $M$ independent of $\theta,b\in \R$ and $t\geq \frac{t_0}{2}=2$. 

Now to obtain the gaussian decay of the density, we will split $J$ in only two pieces.
We shall use the following  direct application  of Lemma 2.1 in \cite{Bobkov:23}. 
Let $X_1$ and $X_2$ two independent random variables with bounded densities:
\[
p_{X_1}\leq M_1  \textrm{ and } p_{X_2}\leq M_2 
\]
and with (sub-Gaussian) Laplace transform:
\[
\E [ e^{tX_1}]  \leq e^{V_1(t)}= C_1 e^{\frac{\sigma_1^2 t^2}{2}} \textrm{ and } \E [ e^{tX_2}]  \leq e^{V_2(t)}= C_2 e^{\frac{\sigma_2^2 t^2}{2}}. 
\]
Then Lemma 2.1 in \cite{Bobkov:23} applied with $n=2$ and $r_1=r_2=2$ gives 
\[
p_{X_1 +X_2} (x) \leq \sqrt {M_1 M_2} e^{-W^*(x)} 
\]
where  $W^*$  is the Legendre transform of 
\[ W(t) := \frac{1}{2} (V_1(2t) + V_2(2 t))= (\sigma_1 ^2+ \sigma_2^2) \frac{(2t)^2}{2} {+} \frac 1 2 \log C_1  {+} \frac 1 2 \log C_2.\]

Computing this Legendre transform gives:
\[
p_{X_1 +X_2} (x)\leq \sqrt {M_1 M_2} \sqrt {C_1 C_2}   e^{  - \frac{x^2} { 8 (\sigma_1^2+ \sigma_2^2)}}.
\]

Let $t\ge t_0=4$, we thus write:
\[ 
J(\theta,t)=\frac1{\sqrt t} \int_0^t \sin \left( B_s\right) ds 
       =X_1 + X_2
 \]   
 with \[
 X_1 =\frac1{\sqrt t} \int_0^{t/2} \sin \left( B_s\right) ds \textrm{ and }
 X_2 =\frac1{\sqrt t} \int_{t/2}^t \sin \left( B_s\right) ds.
 \]
 We now condition on the event $(B_0,B_{ t/ 2},B_t)=(b_0,b_{1/2},b_1)\mod 2\pi$ with $b_0=\theta,b_1=b$ and $b_{1/2} \in[0,2\pi]$. Under this conditioning, $X_1$ and $X_2$  become independent and 
  have the  same law as $ \frac 1 {\sqrt 2} J(a,b,\frac{t}{2})$ for some $a,b \in \R$. In particular, by \eqref{E:4.21}, they have bounded density with a constant independent of $(b_0,b_{1/2},b_1)\in \R^2$.
  Let us now prove that there exists $C,\sigma>0 $ such that 
\begin{equation}
    \label{E15}
    \esp\left[e^{\lambda X_j}\vert (B_0, B_{t/2}, B_t)= (b_0,b_{1/2},b_1) \mod   2\pi\right]\le C e^{\frac{\sigma^2\lambda^2}{2}},\qquad j=1,2,\ t\ge 4, \lambda \in \R.
\end{equation}

By symmetry of $X_j$, $j=1,2$, it is enough to prove it for $\lambda\geq0$. We work only for $j=1$. By Cauchy-Schwarz and cutting again this integral in two, one has 
\begin{align*}
 &\esp\left[e^{\lambda X_1}\vert  (B_0, B_{t/2})= (b_0,b_{1/2}){\rm mod} 2\pi\right]\\&
 \le \esp\left[e^{\frac{2\lambda}{\sqrt{t}}\int_0^{t/4}\sin(B_s) ds}\vert (B_0, B_{t/2})= (b_0,b_{1/2}){\rm mod} 2\pi \right]^{1/2}\\&\times \esp\left[e^{\frac{2\lambda}{\sqrt{t}}\int_{t/4}^{t/2}\sin(B_s) ds}\vert (B_0, B_{t/2})=(b_0,b_{1/2}){\rm mod} 2\pi\right]^{1/2}.
 \end{align*}
 
It is in fact sufficient to establish~\eqref{E15} for  the first term. It will then also holds for the second term by time reversibility.

Now (conditioning only on $B_0=b_0$), by~\eqref{eq:sqrt-dt} we have for $s\in [0,t/4]$
\begin{equation}
    \label{E18}
    \int_0^s \sin(B_r)dr=2\left(\sin(b_0)-\sin(B_s)+\int_0^s\cos(B_r)\, dB_r\right).
\end{equation}
Letting $(\alpha,\gamma)=(b_0,b_{1/2})\ {\rm mod}\ 2\pi$  and conditioning also with respect to $B_{t/2}=\gamma\ {\rm mod} \ 2\pi$ boils down to replacing $dB_r$ by $d\beta_r+b(r,\gamma)\, dr$ with $\beta_r$ a Brownian motion and $\displaystyle b(r,\gamma)=\nabla\ln p_{\Se}\left(t/2-r,\cdot,\gamma\right)(B_r)$,~$p_{\Se}$ being the heat kernel on the circle. This immediately yields $\displaystyle \left|b(r,\gamma)\right|\le \frac{C_2}{t/2-r}$  for some $C_2>0$ independent of $\gamma$. 

In particular for $0\leq r\leq t/4$, $\left|b(r,\gamma)\right|\le \frac{4C_2}{t}$.
Then for $t\ge 1$, denoting by $(B_s^\gamma)_{0\le s\le t/2}$ the Brownian motion started at $\alpha$ and conditioned on $B_{t/2}=\gamma\ {\rm mod} \ 2\pi$: 
\begin{align*}
  &\esp\left[e^{\frac{\lambda}{\sqrt{t}}\int_0^{t/4}\sin(B_s) ds}\vert  (B_0, B_{t/2})=(\alpha,\gamma) {\rm mod}\  2\pi\right]\\
    &=\esp\left[e^{\frac{2\lambda}{\sqrt{t}}\left(\sin(\alpha)-\sin(B_{t/4}^\gamma)+\int_0^{t/4}(\cos(B_s^\gamma)\,d\beta_s+b(s,\gamma)\, ds)\right)}\right]\\
    &\le e^{\frac{\lambda}{\sqrt t}(4+2C_2)}\esp\left[e^{\frac{2\lambda}{\sqrt{t}}\int_0^{t/4}\cos(B_s^\gamma)\,d\beta_s}\right]\le  e^{\lambda(4+2C_2)}e^{\lambda^2/2}
\end{align*}
where for the last inequality we used the fact that $\int_0^{t/4}\cos(\theta+B_s^\gamma)\,d\beta_s=W_{\int_0^{t/4}\cos^2(\theta+B_s^\gamma)ds}$
for some Brownian motion $(W_s)$, so by convexity of exponential 
\[
\esp\left[e^{\frac{2\lambda}{\sqrt{t}}\int_0^{t/4}\cos(\theta+B_s^\gamma)\,d\beta_s}\right]\le \esp\left[e^{\frac{2\lambda}{\sqrt{t}}W_{t/4}}\right]=e^{\lambda^2/2}.
\]
 Thus Inequality \eqref{E15} holds with some constants independent of $(b_0,b_{1/2},b_1)$ and $t \geq 4$. The  above version of Lemma 2.1 in \cite{Bobkov:23} gives the desired gaussian decay for the conditional law of $J(\theta,b,t)$. By integration and since again, the estimate is independent of $(b_0,b_{1/2},b_1)$, the desired gaussian  decay \eqref{E7} follows
for $p_{\theta,b,t}$ itself.
\end{proof}

To prove Lemma \ref{L3}, we need:
      
        \begin{lemma}
        \label{L2}
       { Let $(B_t)_{t\geq 0}$ be a standard Brownian motion.} For any $\alpha \in (0,1)$ and any $a, b\in \R$ such that $|a|\le |b|$, for all $t> 0$
        \begin{equation}
        \label{E21}
        \E\left[ \left| b + B_t\right|^{-\alpha} \right]
        \leq\E\left[ \left| a + B_t\right|^{-\alpha} \right]
        \leq \E\left[ \left| B_t\right|^{-\alpha} \right].
        \end{equation}
        \end{lemma}

        \begin{proof}
        
        Since for $a\in \R$ we have \[\E\left[ \left| a + B_t\right|^{-\alpha} \right]=\E\left[ \left| a - B_t\right|^{-\alpha} \right]=\E\left[ \left| -a + B_t\right|^{-\alpha} \right]\] by symmetry of the law of $B_t$, we can assume that $0\le a\le b$. 
        Letting \[\tau(a,b)=\inf\left\{t\ge 0, \ B_t=-\frac{b+a}2\right\}\] we define, {using strong Markov property}, a new Brownian motion $(B_t')_{t\ge 0}$ with $B_t'=B_t$ if $t\le \tau(a,b)$ and $B_t'=- b-a -B_t$ if $t\ge \tau(a,b)$ and consider the two Brownian motions $B'_t+a$ and $B_t+b$.
     
        For $t\le \tau(a,b)$ we have $B_t\ge -\frac{b+a}2$ implying \[B_t+b\ge\frac{b-a}2\quad\hbox{and}\quad B_t'+a=B_t+a\ge -\frac{b-a}2 \]
        which, together with $ B_t+b\ge B_t+a$, yields $|B_t+b|\ge |B'_t+a|$.
        
        For $t\ge \tau(a,b)$ we have $B_t'+a=-(B_t+b)$ implying $|B_t+b|= |B'_t+a|$. So for all $t\ge 0$, $|B_t+b|\ge |B'_t+a|$.
        We conclude with the equality $\E\left[ \left| a + B_t\right|^{-\alpha} \right]=\E\left[ \left| a + B_t'\right|^{-\alpha} \right]$.
        \end{proof}
        
         \begin{remark}
          Lemma \ref{L2} is not limited to the Gaussian measures. It holds for any translation of distributions with 
        unimodal symmetric density. {This can be shown by coupling two distributions where  the second one is obtained from the first one by a translation by $b-a$. A sketch of proof is given by the picture below. The green part is translated and the red/blue part are reflected. }
        
{\centering \includegraphics[width=.75\linewidth]{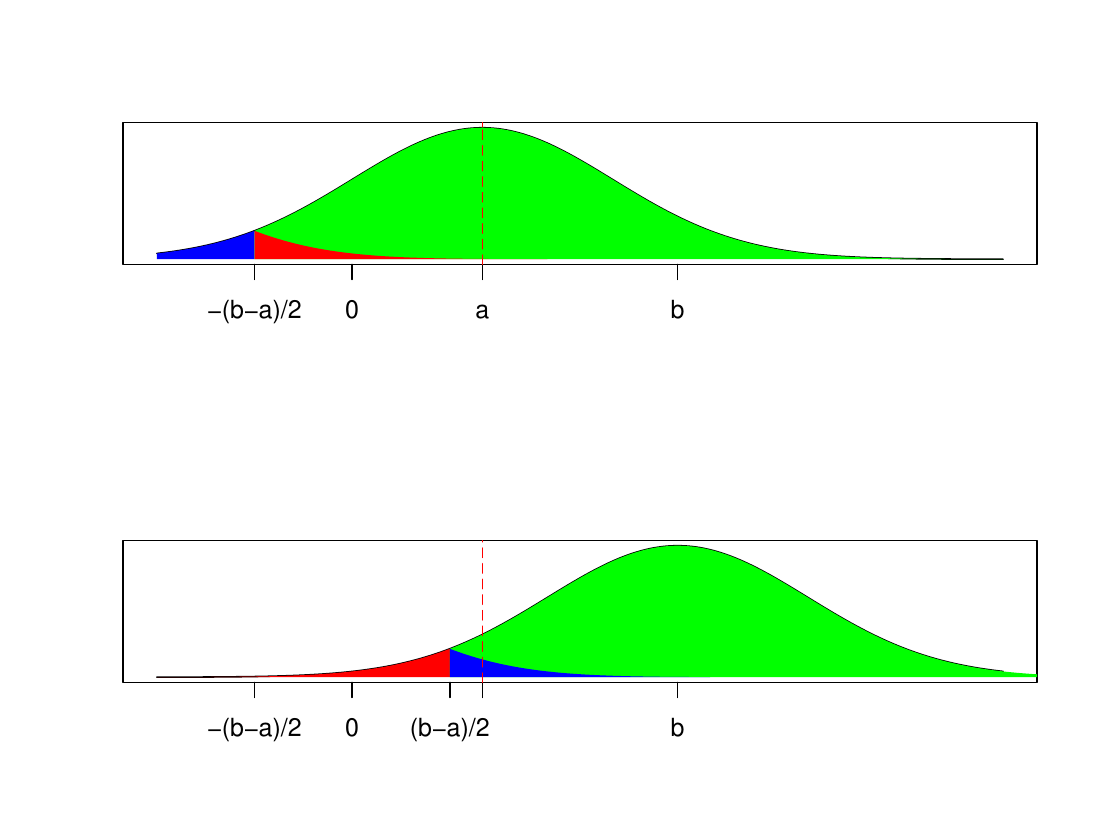} 

}
        \end{remark}
        
        We can finally give the proof of Lemma \ref{L3}.
   \begin{proof}[Proof of Lemma \ref{L3}]
 Note that when $\sin \varphi=1$, we only need the density of $Y$ to be bounded and $0\leq\alpha<1$. Indeed, we directly get:
    \begin{equation}\label{Esinphi}
\esp\left[|Y|^{-\alpha}\right]\le 1+\int_0^1x^{-\alpha}M dx=1+\frac{M}{1-\alpha}.
\end{equation}
For the general case, as $Y$ has a sub-Gaussian density, 
        \begin{equation}\label{EsubG}
     \esp\left[\left|a\cos\varphi+\sin\varphi Y\right|^{-\alpha}\right]\le
     M\sigma\sqrt{2\pi}\esp\left[|a\cos\varphi+\sin\varphi Z|^{-\alpha}\right]]
        \end{equation}
   with $Z\sim \SN(0,\sigma^2)$. From Lemma~\ref{L2}, we know that $\E\left[ \left| a + Z\right|^{-\alpha} \right]$ is nonincreasing in $|a|$. As $a\ge 1$ in \eqref{EsubG}:
   \[\esp\left[\left|a\cos\varphi+\sin\varphi Y\right|^{-\alpha}\right]\le
     M\sigma\sqrt{2\pi}\esp\left[\left|\cos\varphi U+\sin\varphi Z\right|^{-\alpha}\right]
        \]
        where $U$ has uniform law in $[-1,1]$ and is independent of $Z$.
     Now using the inequality $\|f_{X+Y}\|_\infty\le \min(\|f_X\|_\infty\|f_Y\|_1, \|f_Y\|_\infty\|f_X\|_1)=\min(\|f_X\|_\infty, \|f_Y\|_\infty)$ for two independent random variables $X$ and $Y$ with densities $f_X$ and $f_Y$, we get $\displaystyle\|f_{\cos\varphi U+\sin\varphi Z}\|_\infty\le \min\left(\frac{\|f_U\|_\infty}{|\cos\varphi|},\frac{\|f_Z\|_\infty}{|\sin\varphi|}\right)\le \sqrt{2} 
     \left(\frac{1}{2}+ \frac{1}{\sqrt {2\pi} \sigma} \right)$ and this allows us to conclude similarly as in \eqref{Esinphi}.    
            \end{proof}

\begin{proof}[Proof of Lemma \ref{LY}]
For a given $(t_1,t_2\ldots,t_N)\in \prod\limits_{k=1}^N \left[\frac{(k-1/2)T}{N},\frac{kT}{N}\right]$ and $t_0=0$
      we also denote for $1\leq k\leq N$
      \[
      J_k= \frac{1}{\sqrt{t_k-t_{k-1} }}\int_{t_{k-1}}^{t_k} \sin \left(\theta + B_s\right) ds .
      \]
      Our first goal is to bound the density  of the vector $J=(J_1,\dots J_N)$.
      As above, we first consider the conditional density of $J$ with respect to 
      $B_{t_1}=b_1 \mod2\pi,\dots ,  B_{t_N}=b_N\mod2\pi$ for some $b=(b_1, \dots b_N) \in \R^N$. Since $T\geq 1$,
      on this conditioning  the coordinates of the  vector $J$ are  independent and,  by Inequality \eqref{E7} of  Lemma \ref{L1} with {$t_0=\frac{1}{2N}$}, they have sub-Gaussian densities on $\R$, uniform with respect to  $b=(b_1, \dots, b_N)$ and $(t_1,\dots,t_N)$.  
      As a consequence, the conditional density of the random vector  $J$ is bounded by:
      \begin{equation}\label{EfJb}
      f_{J|b\mod(2\pi)} (x) \leq M^N \exp\left(- \frac{x_1^2 +\dots + x_N^2}{2 \sigma^2}\right).
      \end{equation}
Similarly as above,  since the upper bound  does not depend on $b$, it can be integrated in $b$  and  the upper bound in \eqref{EfJb}  is also valid  for the density of the vector $J$ itself.

     We notice now that  $Y$ is obtained as a  linear transformation of $J$, namely
      \[
      {Y_k=}Y(t_k)= \frac{1}{\sqrt {t_k}} 
      \left( \sqrt{t_1} J_1 + \sqrt{t_2 -t_1} J_2+ \dots + \sqrt{t_k -t_{k-1}} J_k \right)
      \]
      and 
      \[
      J_k= \frac{\sqrt{t_k}} {\sqrt{t_k -t_{k-1}}}   Y_k
      - \frac{\sqrt{t_{k-1}}}{\sqrt{t_k -t_{k-1}}} Y_{k-1}{\delta_{k\geq 2}}.
      \]
      We write $Y=AJ$ and $J=A^{-1} Y$, we have
      \begin{align*}
      f_Y (y) &= \frac{1}{\det(A)} f_J(A^{-1}y) \\
             &\leq \frac{1}{\det(A)} M^N \exp \left(- \frac{ \Vert A^{-1} y \Vert^2}{2\sigma^2} \right)
      \end{align*}
      and the upper bound of the density of $Y$ \eqref{EfY} holds since
    \begin{align*}
     \det(A^{-1}) &= \prod_{k=1}^N \frac{\sqrt{t_k}} {\sqrt{t_k -t_{k-1}}} \\
      &\leq \prod_{k=1}^N \frac{\sqrt{k \frac{T}{N}}} {\sqrt{(k-1/2) \frac{T}{N} -(k-1)\frac{T}{N}}}
 = \prod_{k=1}^N \sqrt{2k}
     \end{align*}
     and since the two norms $\Vert \cdot \Vert$ and 
     $\Vert  A^{-1} \cdot \Vert$ in finite dimension are equivalent with, by continuity and compactness, some constants that may be taken independently of  $(t_1,t_2\ldots,t_N)\in \prod_{k=1}^N \left[\frac{(k-1/2)T}{N},\frac{kT}{N}\right]$.
     
Now,   denoting by $Z$ a random vector  with law $ \mathcal N(0, \sigma_N  Id_N)$, one has 
\begin{align*}
    \esp\left[\prod_{k=1}^N  F_k(Y(t_k))\right] 
    &\leq C_N (2\pi \sigma_N^{{2}})^{\frac{N}{2}}
    \esp\left[\prod_{k=1}^N  F_k(Z_k)\right]\\ & = C_N (2\pi \sigma_N^{{2}})^{\frac{N}{2}} 
     \prod_{k=1}^N \esp\left[  F_k(Z_k)\right].
\end{align*}
By  Lemma \ref{L3}, this last quantity is finite.   
The proof of Lemma \ref{LY} is complete.  
\end{proof}

\subsection{Application to the dual}\label{Subsec AsyDual}

{In this last subsection, we use the preceding results to derive the long-time behaviour of the Malliavin dual.}

\begin{proposition}
\label{P5}
Let $q\in[1,\infty)$. There exists $C_q>0$ such that for all $T\geq \sqrt{2}$, $x\in \R\times \CC$ and $v\in \R\times \CC$ ,   the dual $\delta  h$ of the solution $ h=  h(x,v,T)$ to Equation~\eqref{E39} satisfies 
\begin{equation}
    \label{E49bis}
    \esp[|\delta  h|^q]^{\frac{1}{q}}\le C_q\left(|v_\R|+\frac1{\sqrt T}|v_\CC|\right),\qquad v=\begin{pmatrix}
v_\R\\v_\CC
\end{pmatrix}\in \R\times \CC.
\end{equation}
\end{proposition}
 
\begin{proof}
Recall that an explicit expression of $\delta h$ has been computed in Theorem \ref{P4bis} depending on the chosen flow for the kinetic Brownian motion \eqref{E28}. 
Our task is to find an asymptotic for $\delta h$. Following Section~\ref{S3} and Proposition~\ref{P2}, we proceed with renormalisations and use the obtained bounds and convergences. 

Define for $0\leq t,s\leq 1$
\begin{equation}\label{E50bis}
    \tilde\SK_T(t,s):=\begin{pmatrix}
     0 & \sqrt{T}{e^{i\sqrt{T}\tilde B_s}}^{*}\\
    \sqrt{T}e^{i\sqrt{T}\tilde B_s} & -Te^{i\sqrt{T}\tilde B_s}\left(i \int_0^te^{i\sqrt{T}\tilde B_r}dr\right)^{*}-Ti \int_0^te^{i\sqrt{T}\tilde B_r}dr{e^{i\sqrt{T}\tilde B_s}}^{*}
    \end{pmatrix}
    \end{equation}
\begin{equation}
  \text{and }   \label{E51bis} v_T:=\begin{pmatrix}
 v_\R   \\
 \frac{v_\CC}{\sqrt T}
 \end{pmatrix}.
 \end{equation}
  
 Then renormalising \eqref{EDual} we get $-\delta h=\tilde N(T)\begin{pmatrix}
 1 & 0\\
 0 & R(-u)
\end{pmatrix}v_T$ with:
 \begin{align}
  \label{EN1}
  \tilde N(T)&=\begin{pmatrix} \frac{1}{\sqrt{T}}\tilde B_1\\  -i\int_0^1 e^{i\sqrt{T}\tilde B_t}(\tilde B_1-\tilde B_t) dt\end{pmatrix}^{*}\tilde\cC(T)^{-1}\\
   \label{EN2} 
   &+\int_0^1\int_0^t \begin{pmatrix} s\\  \sqrt{T}i\int_0^s e^{i\sqrt{T}\tilde B_{\rho}}(\rho-s)d\rho\end{pmatrix}^{*}\tilde\cC(T)^{-1}\tilde\SK_T(t,s)\tilde\cC(T)^{-1}\,dsdt  
  \end{align}
  and thus by Cauchy-Schwarz inequality in $\R^3$, 
  \[
  |\delta h(\omega)|\le \|\tilde N(T)(\omega)\|\, \|v_T\|,
  \]
  yielding
  \begin{equation}\label{EN2ter}
  \esp[|\delta h|^q]^{\frac{1}{q}}\le  \esp[\|\tilde N(T)\|^q]^{\frac{1}{q}}\, \|v_T\|.
  \end{equation}
Since by Proposition \ref{P2}, the moments of the matrix   $\tilde\cC(T)^{-1}$ are uniformly bounded,
it is direct to see that  the vector given by \eqref{EN1} admits uniformly (with respect to $T\geq \sqrt{2}$) bounded moments of order $q\geq 1$.  
  
  To deal more easily with  the  second vector given by \eqref{EN2},  we write it as  
  \begin{equation}\label{EN2bis}
    \begin{pmatrix}
     1\\ 0\\ 0
    \end{pmatrix}^*\tilde\cC(T)^{-1}\tilde N_1(T)\tilde\cC(T)^{-1}+\begin{pmatrix}
     0\\ 1\\ 0
    \end{pmatrix}^*\tilde\cC(T)^{-1}\tilde N_2(T)\tilde\cC(T)^{-1}+\begin{pmatrix}
     0\\ 0\\ 1
    \end{pmatrix}^*\tilde\cC(T)^{-1}\tilde N_3(T)\tilde\cC(T)^{-1}
 \end{equation} 
 where \begin{align*}
     \tilde N_1(T)&=\int_0^1\int_0^ts\tilde \SK_T(t,s)ds dt\\
     \tilde N_2(T)&=\int_0^1\int_0^t\sqrt{T}\int_0^s-\sin(\sqrt{T}\tilde B_{\rho})(\rho-s)d\rho\tilde\SK_T(t,s)dsdt\\
    \tilde N_3(T)&=\int_0^1\int_0^t\sqrt{T}\int_0^s\cos(\sqrt{T}\tilde B_{\rho})(\rho-s)d\rho\tilde\SK_T(t,s)dsdt.
 \end{align*}
 
 From Lemma \ref{L3.2}, all the coefficients of the matrices $\tilde N_1(T)$, $\tilde N_2(T)$ and $\tilde N_3(T)$ have uniformly  bounded moments of  any (even) order.  Using again Proposition \ref{P2} and  Hölder inequality, we  get
 $\displaystyle
 \sup\limits_{T\geq \sqrt{2}}\esp[\|\tilde N(T)\|^q]^{\frac{1}{q}}<+\infty$
 and thus by~\eqref{EN2ter}
 we can take $$C_p= \sup\limits_{T\geq \sqrt{2}}\esp[\|\tilde N(T)\|^q]^{\frac{1}{q}}$$ to obtain~\eqref{E49bis}.
 \end{proof}
 
 The proof of Proposition~\ref{P5} provides in fact a constant $C_p(T):= \esp[\|\tilde N(T)\|^q]^{\frac{1}{q}}$ for each $T\ge \sqrt{2}$. 
 The next proposition establishes convergence of $\esp[\|\tilde N(T)\|^q]^{\frac{1}{q}}$ as $T\to \infty$, which allows us to  make the estimate in Equation~\eqref{E49bis} precise and explicit in large time.
 \begin{proposition}
  \label{PtildeN}
  When $T$ tends to infinity, the vector $\tilde N(T)$ converges in law to $\tilde N$ defined by
  \begin{equation}\label{E5.8.2}
 \begin{split}
     \tilde N:=\sqrt{2}\int_0^1\int_0^t &\begin{pmatrix}
    s\\  -\sqrt{2}\int_0^s (\rho-s)d{\SW}_{\rho}\end{pmatrix}^{*}
    \cG^{-1}\\
    &
    \qquad\begin{pmatrix}
      0 & \left( d(i{\SW})_s\right)^{*}\\
       d(i{\SW})_s &\sqrt{2}\left( d(i{\SW})_s\left(\SW_t\right)^{*}+\SW_t\left( d(i{\SW})_s\right)^{*}\right)
     \end{pmatrix}\cG^{-1}dt.
     \end{split}
     \end{equation}
  Moreover, for all $q\ge1$, 
  \begin{equation}
      \label{E5.8.1}
     \esp[\|\tilde N(T)\|^q]^{\frac{1}{q}} \xrightarrow[T\to\infty]{}\esp[\|\tilde N\|^q]^{\frac{1}{q}}.
  \end{equation}
 \end{proposition}
 
 \begin{proof}
 We first show that the first term of $\tilde N(T)$ given  in  \eqref{EN1} converges in $L^1$ and thus in law to $0$. As $\tilde{\C}(T)^{-1}$ has bounded moments, it is enough to show that \[\sqrt{T}\int_0^1 e^{i\sqrt{T}\tilde B_t}(\tilde B_1-\tilde B_t) dt\] has bounded moments of order $2$. This ensues from the use of Itô formula: the first term $\sqrt{T}\int_0^1 e^{i\sqrt{T}\tilde B_t}dt\tilde B_1$ is dealt in~\eqref{eq:sqrt-dt}. We also apply Itô formula to the second term  $\sqrt{T}\int_0^1 e^{i\sqrt{T}\tilde B_t}\tilde B_t dt$, where an additional covariation appears:
  $$\frac{-2\tilde B_t}{\sqrt{T}} e^{i\sqrt T  \tilde B_t}   +  \frac{2}{\sqrt{T}} \int_0^t e^{i\sqrt T \tilde B_s} d\tilde B_s +2i  \int_0^t  e^{i\sqrt T \tilde B_s} ds    + 2i\int_0^t \tilde B_s e^{i\sqrt T \tilde B_s} d\tilde B_s.$$

Then,  similarly as  in the proof of Proposition \ref{PG}, since each  coefficient of \eqref{EN2} may be written as a continuous functional of some random vector processes of the form \eqref{ET1}, 
by using Theorem \ref{T1},   $\tilde N(T)$ converges in law to:
 \begin{equation}\label{E54.0}
 \begin{split}
     \tilde N:=
     \sqrt{2}\int_0^1\int_0^t &\begin{pmatrix}
    s\\  -\sqrt{2}\int_0^s (\rho-s)d\SW_{\rho}\end{pmatrix}^{*}
    \cG^{-1}   \\
    &\qquad
    \begin{pmatrix}
      0 & \left(\circ d(i{\SW})_s\right)^{*}   \\
      \circ d(i{\SW})_s &\sqrt{2}\left(\circ d(i{\SW})_s\left(\SW_t\right)^{*}+\SW_t\left(\circ d(i{\SW})_s\right)^{*}\right)
     \end{pmatrix}\cG^{-1}dt
     \end{split}
     \end{equation}
     where $\SW$ is the Brownian motion appearing in $\cG$ (Proposition \ref{PG}) and is independent of $\tilde B$.
     Noticing that $\int_0^s(\rho-s)d\SW_{\rho}=-\int_0^s\SW_{\rho}d\rho$, Itô and Stratonovich integrals are here the same and \eqref{E54.0} finally becomes:
      \begin{equation}\label{E54}
      \begin{split}
     \tilde N=\sqrt{2}\int_0^1\int_0^t &\begin{pmatrix}
    s\\  \sqrt{2}\int_0^s\SW_{\rho}d\rho\end{pmatrix}^{*}
    \cG^{-1}\\&\qquad\begin{pmatrix}
      0 & \left( d(i{\SW})_s\right)^{*}\\
       d(i{\SW})_s &\sqrt{2}\left( d(i{\SW})_s\left(\SW_t\right)^{*}+\SW_t\left( d(i{\SW})_s\right)^{*}\right)
     \end{pmatrix}\cG^{-1}dt.
     \end{split}
     \end{equation}

 Since $\tilde N(T)$ admits moments of any order uniformly bounded in $T\geq \sqrt{2}$, using again  Theorem 3.5. in \cite{Billingsley:99}, we then get $\lim\limits_{T\to\infty}\esp[\|\tilde N(T)\|^q]^{\frac{1}{q}}= \esp[\|\tilde N\|^q]^{\frac{1}{q}}$.
 \end{proof}

\section{Gradient estimates of the kinetic Brownian motion semigroup via Gaussian-space integration by parts}
\label{S4}

In this section we will obtain gradient estimates for the semigroup of the kinetic Brownian motion from the integration by parts formula in Proposition~\ref{P3}, by using the 
convergences in law obtained in Section~\ref{Subsec AsyDual}. Our {first} main result is the following theorem. It is stated for the model on $\R\times \CC$ but, as already explained before, it holds with the same rates and constants for the model on $\bbS\times \CC$.
\begin{theorem}
\label{T2bis}
For every $p\in(1,\infty]$ there exists a constant $C_p>0$ such that for $T$ large enough, $f : \R\times \CC\to [0,\infty]$ measurable { and bounded }, $x\in \R\times \CC$, $\displaystyle v=\begin{pmatrix}
v_\R\\v_\CC
\end{pmatrix}\in \R\times \CC$, 
\begin{equation}
    \label{E47}
    \left|\left(d_xP_Tf,v \right)\right|\le C_p\left(P_T|f|^p(x)\right)^{1/p}\left(\frac1{T^{1/4}}|v_\R|+\frac1{\sqrt T}|v_\CC|\right).
\end{equation}
{(in the case $p=\infty$, $\left(P_T|f|^p(x)\right)^{1/p}$ should be replaced by $\Vert f \Vert_\infty.$) }
\end{theorem}
In the next theorem we improve the horizontal gradient by using first a reflection coupling in $\R$ or in $\bbS$ according to the model, but we need a bound in terms of $\|f\|_\infty$.
\begin{theorem}\label{T3}
\begin{enumerate}
\item There exists a constant $C>0$ such that for all $T\geq 2\sqrt{2}$ and all { bounded measurable } functions $f$ on $\R\times\mathbb{C}$ and all $x\in \R \times \mathbb{C}$, 
\begin{equation}\label{Egbis}
|\partial_u  P_T f(x) | \leq C\frac{1+\ln T}{\sqrt T}  \Vert f \Vert_\infty.
\end{equation}
\item There exists a constant $C>0$ such that for all $T\geq {2\sqrt{2}}$ and all { bounded measurable } function $f$ on ${\mathbb S}\times\mathbb{C}$ and all $x\in {\mathbb S} \times \mathbb{C}$, 
\begin{equation}\label{Egter}
|\partial_u  P_T f(x) | \leq C\frac{1}{\sqrt T}  \Vert f \Vert_\infty.
\end{equation}
\end{enumerate}
\end{theorem}

\begin{proof} [Proof of Theorem \ref{T2bis}]

First consider the vertical gradient, i.e. the case $v_\R=0$.
Equation~\eqref{E47} { for $T\geq \sqrt{2}$} is a consequence of Equation~\eqref{E23} in Proposition~\ref{P3} and then Hölder formula: 
\begin{equation}\label{E48bis}
\begin{split}
   | (d_xP_Tf, v)|&=\left|\E\left[f(X_T^x)\delta   h\right]\right|\\
   &\le \left(P_T|f|^p(x)\right)^{1/p}\|\delta  h\|_q
   \end{split}
\end{equation}
where $q$ is the conjugate of $p$. The proof is an immediate consequence of formula
\begin{equation}
    \label{E48}
 \|\delta  h\|_q\le C_q\left(|v_\R|+\frac1{\sqrt T}|v_\CC|\right)  
\end{equation}
which has been established in Proposition~\ref{P5}.

\smallbreak

Let us now consider the horizontal gradient, i.e. the case $v_{\CC}=0$. We will use Bismut-Elworthy-Li formula together with the first part of the proof. It is enough to do the computation at $x=\binom{0}{0}$.
For simplicity, we will do the proof for $T$ large enough, say $T\geq t_0$ where $t_0$ will {be given below}.

{Similarly as before,} we can assume that $f$ is smooth with compact support. For $u\in \R$, consider the martingale 
\begin{equation}\label{E49}
M_t(u)=P_{T-t}f\left(X_t^{(u,0)}\right)\quad\hbox{satisfying}\quad dM_t(u)= \left(d_{X_t^{(u,0)}}P_{T-t}f,\binom{dB_t}{0}\right).
\end{equation}
Differentiating at $u=0$, using the facts that $\|df\|$ and $\sup_{s\in [0,T]}\|TX_s\|$ are bounded,  we get that $\displaystyle M_t':=M_t'(0)=\left.\frac{d}{du}\right|_{u=0}M_t(u)$ is a bounded martingale. Consequently, for all fixed $\lambda \in [0,T]$, 
\begin{equation}
    \label{E50}
    N_t^\lambda:= {\left(1-\frac{t}\lambda\right)M_t' }+\frac1\lambda \int_0^tM_s'ds, \quad t\in [0,\lambda] 
\end{equation}
is also a martingale.
In this proof, we simply denote $X_t^x$ by $X_t$ (recall $x=\binom{0}{0}$). So we can write 
\begin{align*}
&\left (d_xP_Tf,\binom{1}{0}\right)\\&=M_0'=N_0^\lambda=\E\left[N_\lambda^\lambda\right]=\frac1\lambda\E\left[\int_0^\lambda \left(d_{X_t}P_{T-t}f,TX_t\binom{1}{0}\right)dt\right]\\
&=\frac1\lambda\E\left[\int_0^\lambda \left(d_{X_t}P_{T-t}f,\binom{1}{{i}\int_0^te^{iB_s}ds}\right)dt\right]\\
&=\frac1\lambda\E\left[\int_0^\lambda \left(d_{X_t}P_{T-t}f,\binom{1}{0}\right)dt\right]+\frac1\lambda\E\left[\int_0^\lambda \left(d_{X_t}P_{T-t}f,\binom{0}{{i}\int_0^te^{iB_s}ds}\right)dt\right]\\
&=:I+I\!I.
\end{align*}
Using~\eqref{E49} we see that $I$ is the expectation of the covariation of the two martingales $(P_{T-t}f(X_t))_{0\le t\le \lambda}$ and $\left(\frac1\lambda B_t\right)_{0\le t\le \lambda}$ so that we get 
$\displaystyle 
    I=\E\left[\frac1\lambda B_\lambda P_{T-\lambda}f(X_\lambda)\right]
$
and then
\begin{equation}
    \label{E51}
    |I|\le \left\|\frac{B_\lambda}{\lambda}\right\|_q\E\left[ |P_{T-\lambda}f|^p(X_\lambda)\right]^{1/p}\le \frac1{\sqrt \lambda}\left\|B_1\right\|_q\left( P_{T}(|f|^p)(x)\right)^{1/p}.
\end{equation}
For $I\!I$, we will use the first part. {As when studying the martingales $M_t$ and $M_t'$,} since for all $z\in\CC$ $P_{T-t}f\left(X_t^{(0,z)}\right)$ is a martingale, we can differentiate in $z$ and obtain that for all $v_\CC\in \CC$, 
$\left(d_{{X_t}}P_{T-t}f,TX_t\binom{0}{v_\CC}\right)$ is a (bounded) martingale. But {since $TX_t\binom{0}{v_\CC}=\binom{0}{v_\CC}$,}  $\left(d_{{X_t}}P_{T-t}f,\binom{0}{v_\CC}\right)$ is the same (bounded) martingale. Consequently
\begin{align*}
    I\!I&={\E\left[\frac1\lambda\int_0^\lambda\E\left[\left(d_{X_\lambda}P_{T-\lambda}f,\binom{0}{1}\right)|\mathcal{F}_{t}\right]{i}\left(\int_0^te^{iB_s}ds\right)dt\right]}\\
    &=\E\left[\left(d_{X_\lambda}P_{T-\lambda}f,\binom{0}{\frac1\lambda\int_0^\lambda{i}\left(\int_0^te^{iB_s}ds\right)dt}\right)\right]\\
    &=\E\left[\left(d_{X_\lambda}P_{T-\lambda}f,\binom{0}{\int_0^\lambda{i}\left(1-\frac{s}\lambda\right)e^{iB_s}ds}\right)\right].
\end{align*}
 We additionally suppose that {$T\geq 2\sqrt2 $ and $T\geq 2\lambda$ so that $T-\lambda \ge \sqrt2$}. Now using~\eqref{E48bis} and~\eqref{E48} we get
\begin{equation}
    \label{E55}
  |I\!I|\le \frac{C_p}{\sqrt{T-\lambda}}\E\left[|P_{T-\lambda}f|^p(X_\lambda)\right]^{1/p}\left\|\int_0^\lambda\left(1-\frac{s}\lambda\right)e^{iB_s}ds\right\|_q\le \frac{C_p' \sqrt{2\lambda}}{\sqrt{T}}(P_T(|f|^p)(x))^{1/p}.
\end{equation}
with $\displaystyle C_p'=C_p\sup_{\lambda\ge 1}\frac1{\sqrt \lambda}\left\|\int_0^\lambda\left(1-\frac{s}\lambda\right)e^{iB_s}ds\right\|_q$.
Finally taking $\displaystyle \lambda=\frac{\|B_1\|_q\sqrt{T}}{C_p'\sqrt{2}}$,  minimising the right hand sides of~\eqref{E55} and~\eqref{E51},
we get for {$\displaystyle  T\geq t_0:=\frac{2\|B_1\|^2_q}{C_p'^2}{\wedge 2\sqrt{2}}$} that 
\begin{equation}
    \label{E56}
    \left(d_xP_Tf,\binom{1}{0}\right)\le 2^{\frac{5}{4}}\frac{\sqrt{C_p'\|B_1\|_q}}{T^{1/4}} (P_T(|f|^p)(x))^{1/p}.
\end{equation}

\end{proof}

\begin{remark}\label{NB:lnpt} {Taking conditional expectation, one can replace the dual by its conditional expectation and one  also obtains the derivative formula 
}
\begin{equation}\label{E48ter}
   (d_xP_Tf, v)={-}  \E\left[f(X_T^x)  \E\left[\delta   h (x,T,v) \big| X_T^x\right] \right]
\end{equation}
and it is well known that it corresponds to the gradient of the logarithm of the heat kernel of the semigroup; that is: 
\[
{-}\E\left[\delta   h (x,T,v) \big| X_T^x\right] =  (d_x  \ln p_T( \cdot,X_{{T}}^x), v).
\]
As it is clear by the estimates of Proposition \ref{P5} and the above proof, we believe that such a procedure should improve and provide the optimal rate. Unfortunately, it is in general out of reach to compute directly the conditional expectation of the dual or the derivative of   the heat kernel.  
\end{remark}

\begin{remark}\label{NB2}
Similarly in the above proof, we also have $\displaystyle 
    I=\E\left[\E\left[\frac1\lambda B_\lambda|X_T\right] f(X_T)\right]
$.
The random vector $\displaystyle \left(\frac1\lambda B_\lambda, \frac{1}{\sqrt T}X_T\right)$ is very close in law to a Gaussian vector. If it was  Gaussian with the same expectation and covariance, then it is easy to check that taking $\lambda=1$, we would obtain a bound $\displaystyle \left|\E\left[ B_1|X_T\right]\right|\le \frac{C \|X_T\|}{T}$ which would allow us to obtain an order $T^{-1/2}$ for the horizontal gradient in~\eqref{E47}.
\end{remark}

We now turn to the proof of Theorem \ref{T3}.
\begin{proof}[Proof of Theorem \ref{T3}]
By rotation,  it is enough to consider the horizontal derivative in  $u=0$. We will thus consider the starting points $(u,z)$ and  $(-u,z)$ with $u>0$ small.
The idea is to first perform a reflection coupling of the two driving Brownian motions until they meet. More precisely we construct $X_t^{(u,z)}=(U_t,Z_t)$ and $X_t^{(-u,z)}=(\tilde U_t, \tilde Z_t)$ with 
\[
\left\{ \begin{array}{ccl}
U_t &=& u  + B_t\\
\tilde U_t&=&  -u -B_t = -U_t. 
\end{array}
\right.
\]
We denote by $\tau$ the first  time when $U_t$ and $\tilde U_t$ meet. 
It corresponds to the hitting times of the standard Brownian motion:
\[
\left\{ \begin{array}{ccccl}
\tau&=& \tau_{0}&=& \inf\{t\geq 0, B_t^u= 0 \}  \textrm{ on } \R,\\
\tau&=&\tau_{0,\pi}&=& \inf\{t\geq 0, B_t^u= 0 \textrm{ or } \pi \}  \textrm{ on } \bbS.
\end{array}
\right.
\]
Here $(B_t^u)_{t\geq 0}$ denotes the standard Brownian motion starting in $u$.
Separating the cases where they meet before the time $\frac{T}{2}$, one has
\begin{align*}
    & \left| P_T f(u,z)-P_Tf(-u,z)\right|\\
 =& \left|\E \left[  \left(f\big(X_T^{(u,z)}\big)-f\big(X_T^{(-u,z)}\big)\right) 
 \left( \bone_{\{\tau>\frac{T}{2} \}}+ \bone_{\{\tau\leq \frac{T}{2} \}}\right) \right]\right|  \\
 \leq &\underbrace{ \left|\E \left[  \left(f\big(X_T^{(u,z)}\big)-f\big(X_T^{(-u,z)}\big)\right) 
 \bone_{\{\tau>\frac{T}{2} \}}\right]\right| }_{ =:A}
 + \underbrace{ \left|\E \left[  \left(f\big(X_T^{(u,z)}\big)-f\big(X_T^{(-u,z)}\big)\right)  \bone_{\{\tau\leq \frac{T}{2} \}}\right]\right|}_{=:B}.
\end{align*}
We first consider the case where the coupling does not occur. A  standard estimate {(see~\cite{Lindvall-Rogers})} gives that on $\R$, and thus also on $\bbS$, 
\begin{equation}\label{A}
A\leq 2\Vert f \Vert_\infty \P\left(\tau>\frac T 2 \right) \leq 2\Vert f \Vert_\infty  \frac{u }{\sqrt{2\pi} \sqrt{T/2}}.
 \end{equation}

We now turn to the term where the coupling occurs. 
First we note that 
\[
\left\{ \begin{array}{l }
\tilde U_\tau =U_\tau \in \{0,\pi\}\\
Z_\tau =z +  \int_0 ^\tau e^{i (u+ B_s)} ds\\
\tilde Z_\tau =z +  \int_0 ^\tau e^{-i(u+B_s)} ds\\
\end{array}.
\right.
\]

By using the strong Markov property  and with $\F_{\tau}$ the $\sigma$-field of events up to $\tau$,
one has
\begin{align*}
    B&=
    { \left|\E\Big[ \left( P_{T-\tau} f(U_\tau,Z_\tau) - P_{T-\tau} f(U_\tau,\tilde Z_\tau) \right) \bone_{\{\tau\leq \frac{T}{2}\}} \Big]\right|}\\
    &\leq \E\Big[ \left| P_{T-\tau} f(U_\tau,Z_\tau) - P_{T-\tau} f(U_\tau,\tilde Z_\tau) \right| \bone_{\{\tau\leq \frac{T}{2}\}} \Big].
\end{align*}
Now using {the vertical estimate from Theorem \ref{T2bis}}, one infers:
\begin{align*}
    B&
    \leq \E\left[ \frac{C \Vert f\Vert_\infty}{\sqrt{T-\tau}} \left|Z_\tau - \tilde  Z_\tau \right| \bone_{\{\tau\leq \frac{T}{2}\}} \right]\\
    &\leq  \frac{ \sqrt 2 C \Vert f\Vert_\infty}{\sqrt T } \,
    \E\left[\left|   Z_\tau - \tilde  Z_\tau \right| \bone_{\{\tau\leq \frac{T}{2}\}} \right].
\end{align*}
Note that 
\[
Z_\tau - \tilde  Z_\tau= 2 i \int_0^\tau \sin(B_s^u) ds.
\]
In the case of the circle $\bbS$, one has:
\[
\left|Z_\tau - \tilde  Z_\tau \right|\leq 2 \tau_{0,\pi}
\]
and thus
\[
\E\left[\left|   Z_\tau - \tilde  Z_\tau \right| \bone_{\{\tau\leq \frac{T}{2}\}} \right]
\leq 2 \E[\tau_{0,\pi}]= 2 u (\pi-u) \leq 2\pi u
\]
and the result follows in this situation by dividing by $u$ and letting $u$ go to 0.

In the case of $\R$
\begin{align*}
   & \E\left[ \Big| \int_0^{\tau_0} \sin(B_s^u) ds\Big| \bone_{\{\tau_0\leq \frac{T}{2}\}} \right]\\
  & = \E\left[ \Big| \int_0^{\tau_{0,\pi}} \sin(B_s^u) ds + \int_{\tau_{0,\pi}}^{\tau_0} \sin(B_s^u) ds\Big| \bone_{\{\tau_0\leq \frac{T}{2}\}}  
   \left( \bone_{\{\tau_{0,\pi}=\tau_0\}} + \bone_{\{\tau_{0,\pi}=\tau_\pi\}} \right)\right]\\
  & \leq \E\left[ \tau_{0,\pi} \right] + 
   \E\left[ \Big| \int_{\tau_{0,\pi}}^{\tau_0} \sin(B_s^u) ds\Big| \bone_{\{\tau_0\leq \frac{T}{2}\}}
   \bone_{\{\tau_{0,\pi}=\tau_\pi\}}\right]\\
  &\leq  \E\left[ \tau_{0,\pi} \right] + 
   \E\left[ \Big| \int_{0}^{\tau_0} \sin(B_s^\pi) ds\Big| \bone_{\{\tau_0\leq \frac{T}{2}\}} \right]
   \P\left( \tau_{0,\pi}=\tau_\pi \right).
\end{align*}
As before, one has
\[
\E[\tau_{0,\pi}]= u (\pi-u) \textrm{ and } 
 \P\left( \tau_{0,\pi}=\tau_\pi \right) = \frac{u}{\pi}.
\]
By Itô formula, one has
\[
\int_{0}^{\tau_0} \sin(B_s^\pi) ds = 2 \int_{0}^{\tau_0} \cos(B_s^\pi) dB_s^\pi.
\]

Now, we consider $T/2$ to be in the interval $[2^{k-1},2^k]$ for some $k\geq 1$. We have
\begin{align*}
    &\E\left[ \Big| \int_{0}^{\tau_0} \cos(\pi+ B_s) dB_s\Big| \bone_{\{\tau_0\leq \frac{T}{2}\}} \right]\\
   & \leq 1 + \sum_{j=0}^{k-1} 
    \E\left[ \Big| \int_{0}^{\tau_0} \cos(\pi+ B_s) dB_s\Big| 
    \bone_{\{ 2^j\leq\tau_0\leq 2^{j+1}\}}   \right].
\end{align*}
We now claim that each term of the sum is bounded independently of $j$. By Cauchy-Schwarz and the Itô isometry:
\begin{align*}
    & \E\left[ \Big| \int_{0}^{\tau_0} \cos(\pi+ B_s) dB_s\Big| 
    \bone_{\{ 2^j\leq\tau_0<2^{j+1}\}}   \right] \\
   & \leq  \E\left[ \left( \int_{0}^{\tau_0 \wedge {2^{j+1}}} \cos(\pi+ B_s) dB_s\right)^2   
     \right]^{1/2}
     \P\left(2^j\leq\tau_0<2^{j+1}\right)^{1/2}\\
  & \leq   \E\left[\tau_0 \wedge{2^{j+1}} \right]^{1/2}
     \P\left(2^j\leq\tau_0<2^{j+1}\right)^{1/2}.
\end{align*}
Finally, with $a=2^{j+1}$
\begin{align*}
    \E[\tau_0 \wedge a] 
    &=  \E[\tau_0 \; \bone_{\{\tau_0\leq a\}}] + a \P(\tau_0>a)\\
    &\leq  \int_0^a s \frac{\pi e^{-\frac{\pi^2}{2s}}} {\sqrt {2\pi} s^{3/2}}ds  + a \int_a^\infty  \frac{\pi e^{-\frac{\pi^2}{2s}}} {\sqrt {2\pi} s^{3/2}}ds\\
     &\leq   \sqrt {\frac\pi 2 }\int_0^a \frac{1} {\sqrt {s} }ds  + a  \sqrt {\frac\pi 2 } \int_a^\infty  \frac{1} {s^{3/2}}ds\\
      &\le   C \sqrt a
\end{align*}
and 
\[
\P\left(\frac{a}{2} \leq \tau_0\leq a \right)
\leq   \sqrt {\frac\pi 2 } \int_{\frac {a}{2}}^a   \frac{1} {s^{3/2}}ds \leq \frac{C}{\sqrt a}.
\]
The above claim is thus proven and therefore, for $T/2\in[2^{k-1}, 2^k$], one has
\[
\E\left[ \Big| \int_{0}^{\tau_0} \cos(\pi+ B_s) dB_s\Big| \bone_{\{\tau_0\leq \frac{T}{2}\}} \right]
\leq 1 + C k  \leq 1 +C \frac{\ln T}{\ln 2}.
\]
Finally, putting all the estimates together,
dividing by $u$ and letting $u$ go to 0 gives the result in the case of $\R\times \R^2$.

\end{proof}

We now turn to the Liouville type result for harmonic functions with respect to the  operator $L$ given in \eqref{EL}. 
\begin{theorem}\label{T4}
Any bounded $L$-harmonic function on ${\mathbb S} \times\R^2$ or on  $\R\times \R^2$ is constant.
\end{theorem}
\begin{proof} We do only the proof for $\R\times \R^2$. The other case is similar.
Let $f$ be a  bounded function on $\R\times \CC$ such that $Lf=0$. We thus have $P_tf=f$. Then, from Theorem \ref{T3}, $f$ is constant since:
\begin{align*}
    |(d_xf,v)|=|(d_xP_Tf,v)|\leq C\|f\|_{\infty}\left(\frac{1+\ln(T)}{\sqrt{T}}|v_{\mathbb{R}}|+\frac{1}{\sqrt{T}}|v_{\mathbb C}|\right)\xrightarrow[T\to\infty]{}0.
\end{align*}
\end{proof}

\bigbreak

\noindent{\bf Acknowledgments}
We want to sincerely thank the editor and the three
reviewers for their handling of the review, and their detailed and thoughtful comments
and suggestions which have significantly improved the article.

\bibliographystyle{plain} 
\bibliography{Bibliographie}
\end{document}